\documentclass[11pt,a4paper]{article}
\usepackage[margin=1in]{geometry}
\usepackage{amsmath}
\usepackage{amssymb}
\usepackage{graphicx}
\usepackage{hyperref}
\usepackage{cite} 
\usepackage{authblk}

\usepackage{amsthm}
\usepackage{lastpage}
\usepackage{tikz}
\usepackage{comment}
\usepackage{algorithm}
\usepackage{algpseudocode}
\usepackage{subcaption}

\theoremstyle{definition}
\newtheorem{definition}{Definition}[section]
\newtheorem{problem}{Problem}[section]
\newtheorem{proposition}{Proposition}[section]

\theoremstyle{plain}
\newtheorem{lemma}{Lemma}[section]
\newtheorem{theorem}{Theorem}[section]

\title{Total b-chromatic Colouring of Graphs}
\author[1]{Fabricio Mendoza Granada \thanks{Corresponding author. Supported by a scholarship from the College of Science and Engineering, University of Glasgow, and Becas Carlos Antonio Lopez (BECAL), Paraguay.}\href{mailto:f.mendoza-granada.1@research.gla.ac.uk}{\footnote{f.mendoza-granada.1@research.gla.ac.uk}}}
\author[1]{David Manlove\href{mailto:David.Manlove@glasgow.ac.uk}\footnote{David.Manlove@glasgow.ac.uk}}
        \affil[1]{School of Computing Science, Sir Alwyn Williams Building, University of Glasgow, Glasgow G12 8QQ, UK}
\date{} 

\providecommand{\keywords}[1]
{
  \small	
  \textbf{\textit{Keywords---}} #1
}

\begin{document}
\maketitle

\begin{abstract}
A \emph{b-chromatic colouring} of a graph $G$ is a proper $k$-colouring of the vertices of $G$, for some integer $k$, such that, for each colour $i$ ($1\leq i\leq k$), there exists a vertex $v$ of colour $i$ such that $v$ is adjacent to a vertex of colour $j$, for each $j$ ($1\leq j\leq k$, $j\neq i$).  The \emph{b-chromatic number} of $G$ is the maximum integer $k$ such that $G$ admits a b-chromatic colouring using $k$ colours.  In this paper we introduce the concept of a \emph{total b-chromatic colouring}, which extends the notion of b-chromatic colourings to both vertices and edges in a graph.  We show that the problem of computing the total b-chromatic number is NP-hard in general graphs.  On the other hand for a subclass of caterpillars we give a polynomial-time algorithm to compute the total b-chromatic number, and indeed a total b-chromatic colouring with the maximum number of colours.
\end{abstract}




\section{Introduction}
\label{introduction}
The \emph{b-chromatic colouring} concept, introduced by Irving and Manlove \cite{irving1999b}, has been studied for the past 25 years. The idea of a b-chromatic colouring is built around the concept of a \emph{b-chromatic vertex}. Given a graph $G$ and a  colouring $\mathbf{c}$ of $G$, a b-chromatic vertex is a vertex $v$ that is adjacent to at least one vertex of every colour other than the colour of $v$. A b-chromatic colouring of $G$ is a  colouring such that for every colour $c$ there must exist a b-chromatic vertex of colour $c$. The \emph{b-chromatic number}, denoted by $\varphi(G)$, is the maximum integer $k$ such that $G$ admits a b-chromatic colouring using $k$ colours. A survey on the b-chromatic number of a graph can be found in \cite{jakovac2018b}.

The concept of a b-chromatic colouring can also be motivated through the use of the following colour-suppressing heuristic to approximate a  colouring that uses the minimum number of colours. Start with an arbitrary  colouring of $G$ and try to reduce the number of colours as follows.  For each colour $c$ in turn, check whether each vertex $v$ of colour $c$ has a ``spare'' colour in its neighbourhood; that is, whether $v$ is not already adjacent to a vertex of every other colour.  If this is the case, recolour each vertex of colour $c$ by one of the available colours that is not in its neighbourhood.  Then remove colour $c$ from the available colours. The procedure terminates when there is no colour available to suppress and the resulting colouring is then a b-chromatic colouring of $G$. The b-chromatic number then indicates the worst-case behaviour of this heuristic. The concept of b-chromatic colouring has also been applied to develop a new clustering technique \cite{elghazel2006new} and a new method for \emph{address block localization} in mail sorting systems \cite{gaceb2008improvement}.

From the above heuristic, it is clear that any  colouring of a graph $G$ that uses the minimum number of colours is a b-chromatic colouring. Hence, $\chi(G)\leq \varphi(G)$, where $\chi(G)$ is the chromatic number of $G$. The first upper bound for the b-chromatic number was given by Irving and Manlove \cite{irving1999b}. Suppose that the vertices $v_1,v_2,\dots, v_n$ of $G$ are ordered such that $d(v_1)\geq d(v_2)\geq \dots \geq d(v_n)$, where $d(v_i)$ represents the degree of $v_i$. The \emph{$m$-degree} of $G$, denoted as $m(G)$, is defined as the maximum integer $i$ for which $d(v_i)\geq i-1$. Then, it holds that $\varphi(G)\leq m(G)$ \cite{irving1999b}. Therefore, 
\begin{align}
    \chi(G)\leq \varphi(G) \leq m(G).
\end{align}

Given a graph $G$ and an integer $k$, the {\sc b-chromatic num} problem asks whether $\varphi(G)\geq k$. The {\sc b-chromatic num} problem is NP-complete for general \cite{irving1999b} and bipartite graphs \cite{manlove1998minimaximal,kratochvil2002b}. Furthermore, the problem remains hard for chordal graphs \cite{havet2012b}, the complement of bipartite graphs \cite{bonomo2015b}, and line graphs \cite{campos2015b}. However, computing a b-chromatic colouring with the maximum number of colours is solvable in polynomial time for trees \cite{irving1999b}, split graphs \cite{havet2012b}, $P_4$-free graphs \cite{bonomo2009b} and graphs with large girth \cite{campos2015graphs, kouider2017b}. Moreover, the b-chromatic number of a $d$-regular graph satisfies $\varphi(G)=d+1$ if $|V(G)|\geq 2d^3$ \cite{cabello2011b}. Recently, Jaffke et al.\ \cite{jaffke2024b} have extended the results on $P_4$-free graphs to graphs with bounded \emph{clique-width}. (Note that graphs of bounded treewidth have bounded clique-width, so this result generalises the earlier polynomial-time algorithm for trees.)
Two concepts related to b-chromatic colourings have been studied in the literature recently, namely \emph{z-colourings} \cite{Zaker2020} and \emph{b-Greedy colourings} \cite{CostaHavet2024}.

The \emph{b-chromatic edge-colouring} concept is an extension of b-chromatic colouring where edges are coloured instead of vertices. The \emph{b-chromatic index}, denoted by $\varphi'(G)$, is the maximum integer $k$ such that $G$ admits a b-chromatic edge-colouring using $k$ colours. Given graph $G$ and integer $k$, the {\sc b-chromatic index} problem asks whether $\varphi'(G)\geq k$. Note that a b-chromatic edge-colouring on a graph $G$ is equivalent to a b-chromatic colouring of the line graph of $G$. Hence, the {\sc b-chromatic index} problem is also NP-complete \cite{campos2015b}. Campos et.\ al.\ \cite{campos2018edge} showed that computing a b-chromatic edge colouring using the maximum number of colours is solvable in polynomial time in trees.

The concept of a \emph{total colouring} is an extension of graph colouring where all \emph{elements} of the graph (i.e., all vertices and edges) are coloured, and where no two elements (either vertices or edges) that are adjacent or incident to one another can be given the same colour. A total colouring of a graph $G$ is equivalent to a colouring of the \emph{total graph} $T(G)$ of $G$, which has vertex set $V(G)\cup E(G)$, with two vertices in $T(G)$ being adjacent if and only if the corresponding elements in $G$ are adjacent or incident. The \emph{total chromatic number}, denoted by $\chi_t(G)$, is the minimum integer $k$ such that $G$ admits a total colouring that uses $k$ colours. The {\sc total chromatic num} problem asks whether $\chi_t(G)\leq k$, for a given integer $k$. Sanchez-Arroyo \cite{sanchez1989determining} showed that the {\sc total chromatic num} problem is NP-complete even for cubic bipartite graphs. A natural lower bound for the total chromatic number is $\chi_t(G)\geq \Delta(G)+1$, where $\Delta(G)$ is the maximum degree of $G$. Vizing \cite{vizing1964estimate} and Behzad \cite{behzad1965graphs} independently conjectured that $\Delta(G)+1\leq \chi_t(G)\leq \Delta(G)+2$. A comprehensive survey on total colouring can be found in \cite{geetha2023total}.

In this work, we extend the concept of b-chromatic colourings to the domain of total colourings. In particular, we study \emph{total b-chromatic colourings}, which are built around the concept of a \emph{total b-chromatic element} -- this is an element $x$ that is adjacent or incident to an element of every colour other than the colour of $x$. A total b-chromatic colouring is a total colouring such that for every colour $c$, there exists a total b-chromatic element of colour $c$. Note that a total b-chromatic colouring of $G$ is equivalent to a b-chromatic colouring of $T(G)$. To the best of our knowledge, the concept of a total b-chromatic colouring has not been explicitly defined previously, and has only been studied indirectly in terms of b-chromatic colourings of total graphs that are complete bipartite graphs \cite{morgan13exploring}. 

In Section \ref{section:Preliminaries} we define total b-chromatic colourings and the total b-chromatic number formally. In Section \ref{section:np-completeness} we show that the problem of computing
the total b-chromatic number of a graph is NP-hard. We then restrict the input graph in order to identify tractable cases. In particular, in Section \ref{section:polynomial-time-algorithm-caterpillars} we
show that if $G$ belongs to a class of caterpillars (to be defined), the total b-chromatic number of $G$, and indeed a total b-chromatic colouring of $G$ with the maximum number of colours, can be computed in polynomial time. For space reasons, most of the proofs have been omitted, but can be found in the appendix.


\section{Preliminaries}
\label{section:Preliminaries}
Let $G=(V,E)$ be a simple undirected graph. Given a vertex $u\in V$, $N(u)$ and $E(u)$ will denote the set of vertices and edges adjacent and incident to $u$, respectively. Let $S\subseteq V$ be a subset of vertices of $G$. The \emph{induced} subgraph of $G$ on $S$ is denoted as $G[S]=(S, F)$ where $F=\{(u,v)\in E:u,v\in S\}$. A \emph{path} on $n\geq 1$ vertices in $G$ is a sequence $P_n=\langle w_1,w_2,\dots,w_n \rangle$ of distinct vertices in $V$ such that $(w_i,w_{i+1})\in E$ for $i=1,\dots,n-1$. The \emph{length} of a path $P_n$, denoted by $\ell(P_n)$, is the number of edges in $P_n$, i.e., $\ell(P_n)=n-1$. Whenever $G$ is isomorphic to $P_n$ we will refer to $G$ as a path $P_n$. Furthermore, if the number of vertices in the path $P_n$ is clear from the context, we will refer to $P_n$ as $P$. A \emph{star} is a complete bipartite graph $K_{1,n}$ with a single vertex in one side of the bipartition and $n$ vertices on the other side of the bipartition.

We will call an \emph{element} of $G$ any member $x\in V\cup E$; in this sense, an element of a graph can be a vertex or an edge. Now, let $x,y\in V\cup E$ be two elements of $G$. Elements $x$ and $y$ are \emph{adjacent} if (i) $x,y\in V$ and $(x,y)\in E$, or (ii) $x\in V$ and $y\in E$ and $x$ is an endpoint of $y$, or (iii) $x,y\in E$ and $x$ and $y$ share an endpoint. The \emph{total neighbourhood} of $x$, denoted by $N_t(x)$, is defined as follows: $N_t(x)=\{y\in V\cup E: y \text{ is adjacent to } x\}$. The \emph{total degree} of an element $x\in V\cup E$, denoted by $d_t(x)$, is the cardinality of its total neighbourhood $N_t(x)$. Therefore, the total degree of a vertex $v\in V(G)$ is given by $d_t(v)=2d(v)$, and the total degree of an edge $e\in E$, where $e=(u,v)$, is given by $d_t(e)=d(u)+d(v)$. Next, let $[k]=\{1,\dots,k\}$. We will now define the concept of total colouring in a graph. 


\begin{definition}\label{definition:total--k-colouring}
     A \emph{total  $k$-colouring} of a graph $G=(V,E)$ is a surjective function $\mathbf{c}:V\cup E\rightarrow [k]$, for an integer $k$, such that for every pair of elements $x,y$ in $G$ where $y\in N_t(x)$, $\mathbf{c}(x)\neq \mathbf{c}(y)$.
\end{definition}
When $k$ is clear from the context we will refer to a total $k$-colouring as a total colouring. Given a total colouring $\mathbf{c}:V\cup E\rightarrow [k]$ of $G$, we will say that an element $x$ is ``picking up'' colour $c\in [k]$ if there exists an element $y\in N_t(x)$ such that $\mathbf{c}(y)=c$.\ Definition \ref{definition:total--k-colouring} can be extended to any subset of elements of $G$: for a given $S\subseteq V\cup E$, $\mathbf{c}(S)=\{\mathbf{c}(x):x\in S\}$. We will say that $\mathbf{c}$ is a \emph{partial total colouring} if there exists some uncoloured element $x$ in $G$. Given a graph $G$ and a total colouring of $G$, $\mathbf{c}:V\cup E\rightarrow [k]$ we say that an element $x\in V\cup E$ of colour $c\in [k]$ is a \emph{total b-chromatic element} if $\mathbf{c}(N_t(x))=[k]\setminus \{c\}$. A \emph{total b-chromatic $k$-colouring} of a graph $G$ is a total $k$-colouring $\mathbf{c}$ such that there exists a total b-chromatic element $x$ of colour $c$ for every $c\in [k]$. The \emph{total b-chromatic number} of $G$, denoted by $\varphi_t(G)$, is defined as the maximum integer $k$ for which there exists a total b-chromatic $k$-colouring of $G$. Next, we will give an analogous definition of the $m$-degree for a total colouring of a graph.

\begin{definition}\label{definition:total-m-degree}
    Let $G=(V,E)$ be a graph and let $\Gamma = \{x_1,x_2\dots, x_{|V|+|E|}\}$ be the set of elements of $G$. Assume that elements in $\Gamma$ are sorted such that $d_t(x_1)\geq d_t(x_2)\geq \dots \geq d_t(x_{|V|+|E|})$. The \emph{total $m$-degree} of $G$ is defined as $m_t(G)=\max \{i:d_t(x_i)\geq i-1\}$.\qed
\end{definition}
Definition \ref{definition:total-m-degree} says that the total $m$-degree of a graph $G$ is the maximum integer $i$ such that $G$ contains at least $i$ elements with total degree at least $i-1$. Observe that the total $m$-degree of a graph $G$ gives an upper bound for the total b-chromatic number of the graph.
\begin{lemma}\label{lemma:total-m-degree}
    $\varphi_t(G)\leq m_t(G)$
\end{lemma}
\begin{proof}
    Suppose by contradiction that $\varphi_t(G)\geq m_t(G)+1$. Then there must exist at least $m_t(G)+1$ elements in $G$ adjacent to at least $m_t(G)$ elements of different colours. This implies that each of these $m_t(G)+1$ elements must have total degree at least $m_t(G)$. As $G$ has at least $m_t(G)+1$ elements with total degree at least $m_t(G)$, then the total $m$-degree of $G$ must be at least $m_t(G)+1$, a contradiction.
\end{proof}
It is natural to consider elements with high total degree as ``candidates'' to become total b-chromatic. We will now introduce two definitions to determine when an element is a candidate to become total b-chromatic in a given graph. First, assume that $G$ has total $m$-degree $m_t(G)=m$. A element $x\in V\cup E$ is \emph{total dense} if $d_t(x)\geq m-1$. An element $x\in V\cup E$ is \emph{total tight} if $d_t(x)=m-1$. Observe that if an element $x$ is total tight then it can pick up exactly $m-1$ different colours. Now, given a subset of elements $S\subseteq V\cup E$, we say that $S$ is a set of total dense elements (resp.\ total tight) if for each $x\in S$ it holds that $d_t(x)\geq m-1$ (resp.\ $d_t(x)=m-1)$. Now, we will compute the total b-chromatic number of a path and a star on $n$ vertices.

\begin{proposition}\label{proposition:total-m-degree-paths}
$m_t(P_1)=1,
m_t(P_2)=m_t(P_3)=3, 
m_t(P_4)=4, 
m_t(P_n)=5$ for $n\geq 5$.
\end{proposition}
\begin{proposition}\label{proposition:total-b-chromatic-colouring-paths}
    Let $P$ be a path. Then $\varphi_t(P)=m_t(P)$.
\end{proposition}

\begin{proposition}\label{proposition:total-b-chromatic-colouring-stars}
    $\varphi_t(K_{1,n})=m_t(K_{1,n})=n+1$. 
\end{proposition}

Note that if $T$ consists of a forest of stars $K_{1,n_1}, \dots, K_{1,n_p}$, then $\varphi_t(T)=\max_{1\leq i\leq p} \{\varphi_t(K_{1,n_i})\}$.

\section{NP-completeness in general graphs}
\label{section:np-completeness}

We will now prove that deciding whether $\varphi_t(G)\geq k$ for a given graph $G$ and an integer $k$ is NP-complete. This can be formalised as follows:

\begin{problem}\label{problem:total-b-chromatic-number}
    {\sc total b-chromatic num}

    Input: \emph{Graph $G$ and integer $k$.}
    
    Question: \emph{Is $\varphi_t(G)\geq k$?}\qed
\end{problem}

We will prove that {\sc total b-chromatic num} is NP-complete by reducing from the following NP-complete problem:

\begin{problem}\label{problem:total-colouring-cubic-graph}
{\sc total colouring cubic-bipartite}

    Input: \emph{Cubic bipartite graph $G$.}
    
    Question: \emph{Is there a total colouring of $G$ with $\leq 4$ colours?}\qed
\end{problem}
The {\sc total colouring cubic-bipartite} problem is known to be NP-complete \cite{sanchez1989determining}. The main result of this section is the following:
\begin{theorem}\label{theorem:total-b-chromatic-colouring-np-complete}
    {\sc total b-chromatic num} is \emph{NP-complete}.
\end{theorem}
\begin{proof}
    First, we will prove that $\text{{\sc total b-chromatic num}}\in \text{NP}$. Consider a graph $G=(V,E)$ and a total colouring $\mathbf{c}:V\cup E\rightarrow [k']$ for some integer $k'$. First, check that $k'\geq k$. Then, for every colour class $\mathbf{c}^{-1}(i)$, for $1\leq i\leq k'$, check whether there exists some element $x\in \mathbf{c}^{-1}(i)$ such that $\mathbf{c}(N_t(x))=[k']\setminus \{\mathbf{c}(x)\}$. This can be done in polynomial time so $\text{{\sc total b-chromatic num}}\in \text{NP}$.

    It remains to show that $\text{{\sc total b-chromatic num}}\in \text{NP-hard}$. To do so, we are going to reduce from {\sc total colouring cubic-bipartite} . Let $G$ be an instance of {\sc total colouring cubic} with vertex set $V(G)=\{u_1,\dots,u_n\}$. We will construct a graph $H$ as follows. The vertex set of $H$ is defined as $V(H)=V(G)\cup \{v\}\cup \{v_1,v_2,v_3,v_4\}\cup \{u_i^j:1\leq i\leq 4 \text{ and } 1\leq j\leq n+3\}$. The edge set of $H$ is defined as $E(H)=E(G)\cup \{(u_j,v):1\leq j\leq n)\}\cup \{(v,v_i):1\leq i\leq 4\}\cup \{(v_i,u_i^j):1\leq i\leq 4 \text{ and } 1\leq j\leq n+3\}$. The resulting graph $H$ can be observed in Figure \ref{fig:b-totcol-reduction}, presented in Appendix \ref{appendix:section:proofs-omitted-np-completeness}. We claim that $G$ has a total colouring using 4 colours if and only if $H$ has a total b-chromatic colouring using $n+9$ colours. The proof of this claim can be found in Appendix \ref{appendix:section:proofs-omitted-np-completeness}.
\end{proof}

\section{Polynomial-time algorithm for a class of caterpillars}

\label{section:polynomial-time-algorithm-caterpillars}

In this section we will provide polynomial-time algorithms to compute a total b-chromatic colouring on a class of \emph{caterpillars} (which are a subclass of trees). A \emph{caterpillar} is a tree $T$ with a central path $\mathcal{P}$ such that every leaf node in $T$ is at a distance one from $\mathcal{P}$. In particular every vertex on $\mathcal{P}$ has degree at least 2. Note that $\mathcal{P}$ is empty if $T$ is isomorphic to $K_1$ or $K_2$. First, we will prove that some caterpillars only admit a total b-chromatic colouring using $m_t(T)-1$ colours. Then, we will give some sufficient conditions for a caterpillar $T$ to admit a total b-chromatic $m_t(T)$-colouring. We will now introduce some notation.  Let $T$ be a caterpillar with central path $\mathcal{P}$, and let $P=\langle w_1,\dots,w_k\rangle$ be a subpath of $\mathcal{P}$ with $k$ vertices. The \emph{subcaterpillar} induced by $P$, denoted by $T[P]$, is the induced subgraph $T[S]$, where $S=\bigcup_{i=1}^kN(w_i)$. The next definition is the main tool of our polynomial-time algorithms.

\begin{definition}[Total dense path]\label{definition:total-dense-path}
    Let $G$ be a graph with total $m$-degree $m_t(G)=m$. A path $P=\langle w_1,w_2,\dots, w_k \rangle$ on $k$ vertices in $G$ is a \emph{total dense} path if vertices $w_2,\dots, w_{k-1}$ are total dense and every edge $(w_i,w_{i+1})$ is total dense, for $i=1,\dots,k-1$.\qed
\end{definition}

Let $P=\langle w_1,\dots,w_k \rangle$ be a total dense path on $k$ vertices in $G$, for some $k\geq 1$. The \emph{boundary vertices} of $P$ are $w_1$ and $w_k$. If $k>1$, then the \emph{boundary edges} of $P$ are $(w_1,w_2)$ and $(w_{k-1},w_k)$. Let $q$ be the number of total dense elements of $P$. Since $w_1$ and $w_k$ may not be total dense, it follows that $q\in \{2k-3,2k-2,2k-1\}$. Therefore, we will define three types of total dense paths; type 1: $P$ has $q=2k-1$ total dense elements; type 2: $P$ has $q=2k-2$ total dense elements; and type 3: $P$ has $q=2k-3$ total dense elements. If $P$ is a total dense path of type 2 or 3, then we will let $P'$ be the total dense subpath of $P$ with the maximum number of vertices such that $P'$ is a total dense path of type 1. Let $P$ be a general path now, not necessarily total dense. Now let $Q=\langle z_1,\dots,z_l \rangle$ be a path on $l$ vertices in $G$. We say that paths $P$ and $Q$ are \emph{adjacent} in $G$ if $w_k=z_1$. We denote by  $\langle P,Q\rangle=\langle w_1,\dots,w_k,z_1,\dots z_l\rangle$ the path in $G$ with $P$ and $Q$ concatenated. The proof of the next theorem shows how to compute a total b-chromatic colouring using $\varphi_t(T)$ colours on a caterpillar with $m_t(T)\leq 5$.

\begin{theorem}\label{theorem:total-b-chromatic-colouring-caterpillar-total-m-degree-at-most-5}
    Let $T$ be a connected caterpillar with $m_t(T)\leq 5$. Then, $\varphi_t(T)=m_t(T)$.
\end{theorem}

We will now characterise caterpillars with total $m$-degree $m_t(T)\geq6$ that do not admit a total b-chromatic $m_t(T)$-colouring.

\begin{definition}\label{definition:total-pivoted-caterpillar}
    Let $T$ be a caterpillar with central path $\mathcal{P}$ and total $m$-degree $m_t(T)\geq 6$ and $m$ total dense elements, where $m=m_t(T)$. Then, $T$ is a \textit{total pivoted} if either one of the following conditions holds:
    \begin{enumerate}
        \item there exists a total dense vertex $u$ with degree $d(u)=m-2$ such that $u'\in N(u)$ with $d(u')=2$ is adjacent to a total dense vertex $v\neq u$.
        \item there exist three paths $P_1,P_2,Q\subseteq \mathcal{P}$ such that (i) $P_1$ and $P_2$ are the only total dense paths of $T$, (ii) each of $P_1$ and $P_2$ has exactly three total dense elements, (iii) $Q$ is a path with $\ell(Q)\leq 1$ and no total dense element, and (iv) $Q$ is adjacent to both $P_1$ and $P_2$.
    \end{enumerate}
    We refer to $T$ as a total pivoted caterpillar of type 1 (resp.\ 2) if condition 1 (resp.\ 2) above holds.\qed
\end{definition}

Examples of total pivoted caterpillars of type 1 and 2 are presented in Figures \ref{fig:total-pivoted-caterpillar-type-1} and \ref{fig:total-pivoted-caterpillar-type-2} in Appendix \ref{appendix:section:figures}. The next two results show that total pivoted caterpillars only admit a total b-chromatic $(m_t(T)-1)$-colouring.

\begin{lemma}\label{lemma:total-pivoted-caterpillar-less-than-m}
     Let $T$ be a total pivoted caterpillar. Then, $\varphi_t(T)<m_t(T)$.
\end{lemma}

\begin{theorem}\label{theorem:total-pivoted-caterpillar-m-1}
    Let $T$ be a total pivoted caterpillar. Then, $T$ admits a total b-chromatic $(m_t(T)-1)$-colouring.  Hence $\varphi_t(T)=m_t(T)-1$.
\end{theorem}

Now, it remains to deal with caterpillars that are not total pivoted. First, observe that that the family of total pivoted caterpillars of type 1 and 2 are mutually exclusive.

\begin{lemma}\label{lemma:total-pivoted-catepillar-type-1-2-mutually-exclusive}
    Let $T$ be a total pivoted caterpillar. Then, $T$ is either a total pivoted caterpillar of type 1 or a total pivoted caterpillar of type 2 but not both.
\end{lemma}
\begin{proof}[Proof sketch]
    Suppose that $T$ is a total pivoted caterpillar of type 1 and 2. Let $P_1$ and $P_2$ be the two total dense paths of $T$. Next, let $u\in V(T)$ be a vertex of $T$ with $d(u)= m-2\geq 4$. Observe now that every edge in $E(u)$ is a total dense path. It follows that $T$ has at least three total dense paths, a contradiction.
\end{proof}

By Lemma \ref{lemma:total-pivoted-catepillar-type-1-2-mutually-exclusive}, if $T$ has a total dense element outside the central path $\mathcal{P}$, then $T$ cannot be a total pivoted caterpillar of type 2. Furthermore, if $T$ has no total dense element outside the central path $\mathcal{P}$, then $T$ cannot be a total pivoted caterpillar of type 1. We will present two algorithms for computing the total b-chromatic number of a non-total pivoted caterpillar according to whether or not there exists total dense elements outside the central path. The first algorithm is given in the proof of the following theorem.

\begin{theorem}\label{theorem:non-total-pivoted-caterpillar-total-element-outside-central-path}
    Let $T$ be a caterpillar with central path $\mathcal{P}$ and $m_t(T)\geq 6$ such that $T$ is not total pivoted. Assume that there exists a total dense element outside $\mathcal{P}$. Then, $\varphi_t(T)=m_t(T)$.
\end{theorem}

\begin{algorithm}[t!]
    \caption{Algorithm for colouring $T[P]$, where $P=\langle w_1,w_2,\dots, w_k \rangle$ and $k\geq 3$, using $2k-1$ colours}\label{algorithm:total-colouring-total-dense-path-type-1}
    \begin{algorithmic}[1]
        \State $\mathbf{c}(w_i)\gets 2i-1$, for $i\in [k]$\label{alg1:line:1-colouring-vertices}
        \State $\mathbf{c}(e_i)\gets 2i$, for $i\in [k-1]$\label{alg1:line:2-colouring-edges}\Comment{$e_i=(w_i,w_{i+1})$ for $i\in [k-1]$}
        \For{$j=1,\dots k-2$}\label{alg1:line:3-for-loop-to-colour-N_1E_1-N_kE_k}\Comment{$w_i^j$ are the neighbours of $w_i$ outside $P$, where $i\in\{1,k\}$ and $j\in [k-2]$}
            \If{$j\mod 2=1$}\label{alg1:line:4-if-j-is-odd}
                \State $\mathbf{c}\left(w_1^j\right)\gets 2j+2$ and $\mathbf{c}\left(e_1^j\right)\gets 2j+3$\label{alg1:line:5-colouring-w_1^j-e_1^j-odd}\Comment{$w_1$ is picking up $\{4,5,8,9,\dots\}$}
                \State $\mathbf{c}\left(w_k^{k-j-1}\right)\gets 2k-2j-2$ and $\mathbf{c}\left(e_k^{k-j-1}\right)\gets 2k-2j-3$\label{alg1:line:6-colouring-w_k^j-e_k^j-odd}\Comment{$w_k$ is picking up $\{2k-4,2k-5,\dots\}$}
            \Else\label{alg1:line:7}
                \State $\mathbf{c}\left(w_1^j\right)\gets 2j+3$ and $\mathbf{c}\left(e_1^j\right)\gets 2j+2$\label{alg1:line:8-colouring-w_1^j-e_1^j-even}\Comment{$w_1$ is picking up $\{6,7,10,11,\dots\}$}
                \State $\mathbf{c}\left(w_k^{k-j-1}\right)\gets 2k-2j-3$ and $\mathbf{c}\left(e_k^{k-j-1}\right)\gets 2k-2j-2$\label{alg1:line:9-colouring-w_k^j-e_k^j-even}\Comment{$w_k$ is picking up $\{2k-6,2k-7,\dots\}$}
            \EndIf\label{alg1:line:10}
        \EndFor\label{alg1:line:11}
        \For{$i=2,\dots,k-1$}\label{alg1:line:12-for-loop-to-colour-N_iE_i}\Comment{$w_i^j$ are the neighbours of $w_i$ outside $P$, where $i=2,\dots,k-1$ and}
        \State \Comment{$j\in[k-3]$. An illustration of this \emph{for-loop} can be observed in Figure \ref{fig:total-b-chromatic-colouring-caterpillar-path-P'}}
            \For{$j=1,\dots,i-2$}\label{alg1:line:13-for-loo-to-colour-w_i^j-e_i^j-up-to-i-2}
                    \If{$j\mod 2=1$}\label{alg1:line:14-if-j-odd}
                        \State $\mathbf{c}\left(w_i^{i-j-1}\right)\gets 2i-2j-2$ and $\mathbf{c}\left(e_i^{i-j-1}\right)\gets 2i-2j-3$\label{alg1:line:15-colouring-w_i^j-e_i^j-odd}\Comment{$w_i$ is picking up $\{2i-4,2i-5,\dots\}$}
                    \Else\label{alg1:line:16-if-j-even}
                        \State $\mathbf{c}\left(w_i^{i-j-1}\right)\gets 2i-2j-3$ and $\mathbf{c}\left(e_i^{i-j-1}\right)\gets 2i-2j-2$\label{alg1:line:17-colouring-w_i^j-e_i^j-even}\Comment{$w_i$ is picking up $\{2i-6,2i-7,\dots\}$}
                    \EndIf\label{alg1:line:18}
            \EndFor\label{alg1:line:19}
            \For{$j=1,\dots,k-i-1$}\label{alg1:line:20-for-loo-to-colour-w_i^j-e_i^j-from-to-i-1}
                    \If{$j\mod 2=1$}\label{alg1:line:21-if-j-is-odd}
                        \State $\mathbf{c}\left(w_i^{i+j-2}\right)\gets 2i+2j$ and $\mathbf{c}\left(e_i^{i+j-2}\right)\gets 2i+2j+1$\label{alg1:line:22-colouring-w_i^j-e_i^j-odd}\Comment{$w_i$ is picking up $\{2i+2,2i+3,\dots\}$}
                    \Else\label{alg1:line:23}
                        \State $\mathbf{c}\left(w_i^{i+j-2}\right)\gets 2i+2j+1$ and $\mathbf{c}\left(e_i^{i+j-2}\right)\gets 2i+2j$\label{alg1:line:24-colouring-w_i^j-e_i^j-even}\Comment{$w_i$ is picking up $\{2i+4,2i+5,\dots\}$}
                    \EndIf\label{alg1:line:25}
            \EndFor\label{alg1:line:26}
        \EndFor\label{alg1:line:27}
    \end{algorithmic}
\end{algorithm}

Suppose now that $T$ has every total dense element within the central path. Observe that $T$ may have more than one total dense path. We will now present an algorithm that computes a total b-chromatic colouring using the maximum number of colours whenever $T$ has exactly one total dense path. This algorithm can be used to compute a total b-chromatic colouring when $T$ has more than one total dense path provided that $T$ is not total pivoted of type 2. Now, let $P\subseteq \mathcal{P}$ be the only total dense path of $T$ and assume that $P=\langle w_1,\dots,w_k\rangle$. Observe that $T$ is not total pivoted. Assume that $P$ has $q$ total dense elements and that $m=m_t(T)\geq 6$. We remark that the algorithms presented in this section work for any $q\leq m$, but to obtain the main result of this section we will assume that $q=m$. Since $P$ contains all the total dense elements of $T$, it follows that $k\geq 4$. We will first show that it is possible to assign $q$ different colours to the total dense elements of $T[P]$ such that each total dense element of $T[P]$ picks up $q-1$ different colours, and no total tight element in $P$ repeats a colour. Recall that $P$ can be a total dense path of type 1, 2 or 3. Assume first that $P$ is a total dense path of type 1. It follows that $q=2k-1$. Algorithm \ref{algorithm:total-colouring-total-dense-path-type-1} constructs a partial total colouring of $T$ by assigning colours to $T[P]$. A pictorial representation of the colouring produced by Algorithm \ref{algorithm:total-colouring-total-dense-path-type-1} is given in Figure \ref{fig:total-b-chromatic-colouring-caterpillar-path-P'}.
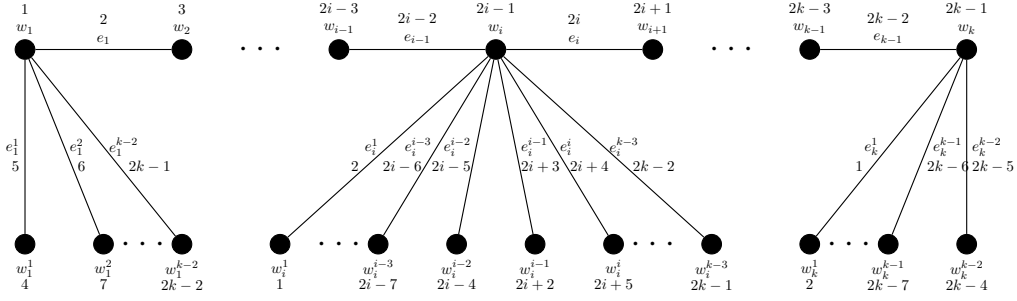
\begin{figure}
    \centering
    \resizebox{.85\linewidth}{!}{
        \begin{tikzpicture}[thick]
            \node[draw, circle, fill=black, minimum size=0.5cm, inner sep=0pt, label=above:{$w_1$}, label={[yshift=0.5cm]above:{1}}] (w1) at (0,0) {};
            \node[draw, circle, fill=black, minimum size=0.5cm, inner sep=0pt, label=above:{$w_2$}, label={[yshift=0.5cm]above:{$3$}}] (w2) at (4,0) {};
            \node[draw, circle, fill=black, minimum size=0.5cm, inner sep=0pt, label=above:{$w_{i-1}$}, label={[yshift=0.5cm]above:{$2i-3$}}] (wi_minus_1) at (8,0) {};
            \node[draw, circle, fill=black, minimum size=0.5cm, inner sep=0pt, label=above:{$w_i$}, label={[yshift=0.5cm]above:{$2i-1$}}] (wi) at (12,0) {};
            \node[draw, circle, fill=black, minimum size=0.5cm, inner sep=0pt, label=above:{$w_{i+1}$}, label={[yshift=0.5cm]above:{$2i+1$}}] (wi_plus_1) at (16,0) {};
            \node[draw, circle, fill=black, minimum size=0.5cm, inner sep=0pt, label=above:{$w_{k-1}$}, label={[yshift=0.5cm]above:{$2k-3$}}] (wk_minus_1) at (20,0) {};
            \node[draw, circle, fill=black, minimum size=0.5cm, inner sep=0pt, label=above:{$w_k$}, label={[yshift=0.5cm]above:{$2k-1$}}] (wk) at (24,0) {};
        
            \node[draw, circle, fill=black, minimum size=0.5cm, inner sep=0pt, label=below:{$w_1^1$}, label={[yshift=-0.5cm]below:{$4$}}] (w1_1) at (0,-5) {};
            \node[draw, circle, fill=black, minimum size=0.5cm, inner sep=0pt, label=below:{$w_1^2$}, label={[yshift=-0.5cm]below:{$7$}}] (w1_2) at (2,-5) {};
            \node[draw, circle, fill=black, minimum size=0.5cm, inner sep=0pt, label=below:{$w_1^{k-2}$}, label={[yshift=-0.5cm]below:{$2k-2$}}] (w1_k_minus_2) at (4,-5) {};
        
            \node[draw, circle, fill=black, minimum size=0.5cm, inner sep=0pt, label=below:{$w_i^1$}, label={[yshift=-0.5cm]below:{$1$}}] (wi_1) at (6.5,-5) {};
            \node[draw, circle, fill=black, minimum size=0.5cm, inner sep=0pt, label=below:{$w_i^{i-3}$}, label={[yshift=-0.5cm]below:{$2i-7$}}] (wi_i_minus_3) at (9,-5) {};
            \node[draw, circle, fill=black, minimum size=0.5cm, inner sep=0pt, label=below:{$w_i^{i-2}$}, label={[yshift=-0.5cm]below:{$2i-4$}}] (wi_i_minus_2) at (11,-5) {};
            \node[draw, circle, fill=black, minimum size=0.5cm, inner sep=0pt, label=below:{$w_i^{i-1}$}, label={[yshift=-0.5cm]below:{$2i+2$}}] (wi_i_minus_1) at (13,-5) {};
            \node[draw, circle, fill=black, minimum size=0.5cm, inner sep=0pt, label=below:{$w_i^{i}$}, label={[yshift=-0.5cm]below:{$2i+5$}}] (wi_i) at (15,-5) {};
            \node[draw, circle, fill=black, minimum size=0.5cm, inner sep=0pt, label=below:{$w_i^{k-3}$}, label={[yshift=-0.5cm]below:{$2k-1$}}] (wi_k_minus_3) at (17.5,-5) {};
        
            \node[draw, circle, fill=black, minimum size=0.5cm, inner sep=0pt, label=below:{$w_k^1$}, label={[yshift=-0.5cm]below:{$2$}}] (wk_1) at (20,-5) {};
            \node[draw, circle, fill=black, minimum size=0.5cm, inner sep=0pt, label=below:{$w_k^{k-1}$}, label={[yshift=-0.5cm]below:{$2k-7$}}] (wk_k_minus_1) at (22,-5) {};
            \node[draw, circle, fill=black, minimum size=0.5cm, inner sep=0pt, label=below:{$w_k^{k-2}$}, label={[yshift=-0.5cm]below:{$2k-4$}}] (wk_k_minus_2) at (24,-5) {};
        
            \draw[thick] (w1) -- (w2) node[midway, above, yshift=0.5cm] {$2$} node[midway, above] {$e_1$};
            \draw[thick] (wi_minus_1) -- (wi) node[midway, above, yshift=0.5cm] {$2i-2$} node[midway, above] {$e_{i-1}$};
            \draw[thick] (wi) -- (wi_plus_1) node[midway, above, yshift=0.5cm] {$2i$} node[midway, above] {$e_{i}$};
            \draw[thick] (wk_minus_1) -- (wk) node[midway, above, yshift=0.5cm] {$2k-2$} node[midway, above] {$e_{k-1}$};
        
            \draw[thick] (w1) -- (w1_1) node[midway, left, yshift=-0.5cm] {$5$} node[midway, left] {$e_1^1$};
            \draw[thick] (w1) -- (w1_2) node[midway, right, yshift=-0.5cm, xshift=.2cm] {$6$} node[midway, right] {$e_1^2$};
            \draw[thick] (w1) -- (w1_k_minus_2) node[midway, right, yshift=-0.5cm, xshift=.5cm] {$2k-1$} node[midway, right] {$e_1^{k-2}$};
            \draw[thick] (wi) -- (wi_1) node[midway, left, yshift=-0.5cm, xshift=-0.6cm] {$2$} node[midway, left, xshift=-.1cm] {$e_i^1$};
            \draw[thick] (wi) -- (wi_i_minus_3) node[midway, left, yshift=-0.5cm, xshift=-.25cm] {$2i-6$} node[midway, left] {$e_i^{i-3}$};
            \draw[thick] (wi) -- (wi_i_minus_2) node[midway, left, yshift=-0.5cm] {$2i-5$} node[midway, left] {$e_i^{i-2}$};
            \draw[thick] (wi) -- (wi_i_minus_1) node[midway, right, yshift=-0.5cm] {$2i+3$} node[midway, right] {$e_i^{i-1}$};
            \draw[thick] (wi) -- (wi_i) node[midway, right, yshift=-0.5cm, xshift=.25cm] {$2i+4$} node[midway, right] {$e_i^i$};
            \draw[thick] (wi) -- (wi_k_minus_3) node[midway, right, yshift=-0.5cm, xshift=0.6cm] {$2k-2$} node[midway, right] {$e_i^{k-3}$};
            \draw[thick] (wk) -- (wk_1) node[midway, left, yshift=-0.5cm, xshift=-.5cm] {$1$} node[midway, left, xshift=-.1cm] {$e_k^1$};
            \draw[thick] (wk) -- (wk_k_minus_1) node[midway, right, yshift=-0.5cm,xshift=-.15cm] {$2k-6$} node[midway, right] {$e_k^{k-1}$};
            \draw[thick] (wk) -- (wk_k_minus_2) node[midway, right, yshift=-0.5cm] {$2k-5$} node[midway, right] {$e_k^{k-2}$};
        
            \node at (6, 0) {\Huge$\cdot\cdot\cdot$};
        
            \node at (18, 0) {\Huge$\cdot\cdot\cdot$};
        
            \node at (3, -5) {\Huge$\cdot\cdot\cdot$};
            \node at (8, -5) {\Huge$\cdot\cdot\cdot$};
            \node at (16, -5) {\Huge$\cdot\cdot\cdot$};
            \node at (21, -5) {\Huge$\cdot\cdot\cdot$};
        
        \end{tikzpicture}
    }
    \caption{Note that since $k$ is assumed to be odd and $i$ is assumed to be even $\mathbf{c}(w_i^1)=1$ and $\mathbf{c}(w_i^{k-3})=2k-1$.}
    \label{fig:total-b-chromatic-colouring-caterpillar-path-P'}
\end{figure}
For the special case where $k=6$, the colouring produced by Algorithm \ref{algorithm:total-colouring-total-dense-path-type-1} is presented in Figure \ref{fig:total-b-chromatic-colouring-running-example-algorithm}, Appendix \ref{appendix:section:figures}. The following lemma proves that each total dense element in $P$ picks up $2k-2$ different colours. Moreover, it proves that no total dense element is repeating a colour.

\begin{lemma}\label{lemma:total-colouring-total-dense-subpath-type-1}
    Algorithm \ref{algorithm:total-colouring-total-dense-path-type-1} assigns exactly $2k-1$ colours to elements of $T[P]$ such that (i) the total dense elements in $P$ are given colours $1,2,\dots,2k-1$, and (ii) every total dense element in $P$ picks up $2k-2$ colours from the set $[2k-1]$.
\end{lemma}

Note that Algorithm \ref{algorithm:total-colouring-total-dense-path-type-1} works for any total dense path of type 1 with at least 3 vertices. Algorithm \ref{algorithm:total-colouring-total-dense-path-type-2} and \ref{algorithm:total-colouring-total-dense-path-type-3}, presented in Appendix \ref{appendix:section:algorithms}, compute a partial total colouring of $T$ by assigning colours to $T[P]$ when $P$ is a total dense path of type 2 and 3, respectively. The next two lemmas shows that every total dense element in a total dense path of type 2 and type 3 can pick up $2k-3$ and $2k-4$ colours, respectively.

\begin{lemma}\label{lemma:total-colouring-total-dense-subpath-type-2}
    Algorithm \ref{algorithm:total-colouring-total-dense-path-type-2} assigns exactly $2k-2$ colours to elements of $T[P]$ such that (i) the total dense elements in $P$ are given colours $1,2,\dots,2k-2$, and (ii) every total dense element in $P$ picks up $2k-3$ colours from the set $[2k-2]$.
\end{lemma}

\begin{lemma}\label{lemma:total-colouring-total-dense-subpath-type-3}
    Algorithm \ref{algorithm:total-colouring-total-dense-path-type-3} assigns exactly $2k-3$ colours to elements of $T[P]$ such that (i) the total dense elements in $P$ are given colours $1,2,\dots,2k-3$, and (ii) every total dense element in $P$ picks up $2k-4$ colours from the set $[2k-3]$.
\end{lemma}

Now we are ready to show that if $T$ is not a total pivoted tree of type 2 with exactly one central path then $\varphi_t(T)=m_t(T)$. The next theorem states the main result of this section.

\begin{theorem}
    Let $T$ be a caterpillar with central path $\mathcal{P}$ and $m_t(T)\geq 6$ such that $T$ is not total pivoted. Assume that every total dense element is in $\mathcal{P}$ and that $T$ has only one central path. Then, $\varphi_t(T)=m_t(T)$.
\end{theorem}
\begin{proof}[Proof sketch]
    Let $P\subseteq \mathcal{P}$ be the only total dense path of $T$ and assume that $m_t(T)=m$. If $P$ has $q>m$ total dense elements then we can always take a subpath of $P$ with exactly $q=m$ total dense elements. Therefore, we will assume that $P$ has exactly $m$ total dense elements. If $P$ is of type 1 then, by Lemma \ref{lemma:total-colouring-total-dense-subpath-type-1} every total dense element of $P$ can become total b-chromatic. If $P$ is of type 2 then, by Lemma \ref{lemma:total-colouring-total-dense-subpath-type-2} every total dense element of $P$ can become total b-chromatic. Otherwise, $P$ is of type 3 and by Lemma \ref{lemma:total-colouring-total-dense-subpath-type-3}, every total dense element of $T$ can become total b-chromatic. It remains to colour every other element in $T$. We will first colour the uncoloured vertices of $T$ and then the uncoloured edges of $T$. Since every total dense element is in $\mathcal{P}$, it follows that every vertex $w\in V(T)$ has $d(w)\leq m-3$. Therefore, there exists a colour for every pair of adjacent vertices. Next, observe that every edge $e\in \mathcal{P}$ is adjacent or incident to at most 4 elements in the central path. Since $m\geq 6$, it follows that there is a spare colour for $e$. Next, observe that every edge $e\not \in \mathcal{P}$ is incident to at most $m-4$ edges and adjacent to 2 vertices. It follows that $e$ is adjacent or incident to at most $m-2$ other elements. Therefore, there is a spare colour for $e$. Hence, $\varphi_t(T)=m_t(T)$.
\end{proof}

\section{Concluding remarks}\label{section:concluding-remarks}
We proved that computing $\varphi_t(G)$ is NP-hard in general graph. We believe that our proof can be adapted for bipartite graphs. Furthermore, we presented a sufficient condition for a caterpillar $T$ to have $\varphi_t(T)=m_t(T)-1$. Moreover, we presented a couple of sufficient conditions for a caterpillar $T$ to have $\varphi_t(T)=m_t(T)$. These sufficient conditions exploit the structure of caterpillars to obtain a polynomial-time algorithm that computes a total b-chromatic colouring of $T$ using $\varphi_t(T)$ colours. We believe that Algorithms \ref{algorithm:total-colouring-total-dense-path-type-1}, \ref{algorithm:total-colouring-total-dense-path-type-2} and \ref{algorithm:total-colouring-total-dense-path-type-3} can be used to compute a total b-chromatic colouring that uses the maximum number of colours for a non-total pivoted caterpillar with more than one total dense path. The ideas presented in Definition \ref{definition:total-pivoted-caterpillar} can also be used to define the concept of a total pivoted tree. In this sense, a total pivoted tree $T$ will have $\varphi_t(T)\leq m_t(T)-1$. We conjecture that a non-total pivoted tree $T$ satisfies $\varphi_t(T)=m_t(T)$.

\section*{Acknowledgements}

We would like to thank the Programme Committee and the anonymous reviewers for their valuable comments, which helped to improve the presentation of this paper. The first author is supported by a scholarship from the College of Science and Engineering, University of Glasgow, and Becas Carlos Antonio Lopez (BECAL), Paraguay. For the purpose of open access, the authors have applied a Creative Commons Attribution (CC BY) licence to any Author Accepted Manuscript version arising from this submission.



\bibliographystyle{elsarticle-num}
\bibliography{bibliography}

\newpage
\appendix
\section{Proofs omitted in Section \ref{section:Preliminaries}} \label{appendix:section:proofs-omitted-preliminaries}

\begin{proof}[Proof of Proposition \ref{proposition:total-m-degree-paths}]
    First, observe that $P_1$ consists on a single vertex. It follows that there exists at least one element in $P_1$ with total degree at least 0. Hence, $m_t(P_1)=1$. Next, $P_2$ and $P_3$ contains 2 and 3 vertices, respectively. It follows that $P_2$ and $P_3$ contains at least 3 elements with total degree at least 2, respectively. Hence, $m_t(P_2)=m_t(P_3)=3$. Next, $P_4$ contains 4 vertices. Therefore, $P_n$ contains at least 4 elements with total degree at least 3. Hence, $m_t(P_4)=4$. Lastly, $P_n$, for $n\geq 5$, contains at least 5 elements. Observe that every element in $P_n$ has total degree at most 4. Moreover, $P_n$ contains at least three vertices with total degree 4, and at least two edges with total degree 4. Therefore, $P_n$ contains at least 5 elements with total degree at least 5. Hence, $m_t(T)=5$.
\end{proof}

\begin{proof}[Proof of Proposition \ref{proposition:total-b-chromatic-colouring-paths}]
    We will construct a total b-chromatic $m_t(P_n)$-colouring $\mathbf{c}$ of $P_n$, for $n\geq 1$. Assume that $P_n=\langle w_1,\dots,w_n\rangle$. Assume first that $n=1$. Then, assign colour $\mathbf{c}(w_1)\gets 1$ and observe that $\varphi_t(P_1)=1$. Next, assume that $n=2$ and assign colours $\mathbf{c}(w_1)\gets 1, \mathbf{c}(w_1,w_2)\gets 2$ and $\mathbf{c}(w_2)\gets 3$. Note that every element of colour $i\in \{1,2,3\}$ in $P_2$ is picking up colours $\{1,2,3\}\setminus \{i\}$. Hence, $\varphi_t(P_2)=3$. Next, assume that $n=3$ and assign colours $\mathbf{c}(w_1)\gets 1, \mathbf{c}(w_1,w_2)\gets 2, \mathbf{c}(w_2)\gets 3, \mathbf{c}(w_2,w_3)\gets 1$ and $\mathbf{c}(w3)\gets 2$. Note that every element of colour $i\in \{1,2,3\}$ in $P_2$ is picking up colours $\{1,2,3\}\setminus \{i\}$. Hence, $\varphi_t(P_3)=3$. Assume now that $n=4$ and assign colours $\mathbf{c}(w_1)\gets 1, \mathbf{c}(w_1,w_2)\gets 2, \mathbf{c}(w_2)\gets 3, \mathbf{c}(w_2,w_3)\gets 4$ and $\mathbf{c}(w3)\gets 1, \mathbf{c}(w_3,w_4)\gets 2$ and $\mathbf{c}(w_4)\gets 4$. Note that $(w_1,w_2), w_3, (w_2,w_3)$ and $w_4$ are total b-chromatic elements of colours 1,2,3 and 4, respectively. Hence, $\varphi_t(P_4)=4$. 
    
    Now, assume that $n\geq 5$. Observe that $d_t(w_2)=d_t(w_3)=d_t(w_4)=4$ and $d_t(w_2,w_3)=d_t(w_3,w_4)=4$. Next, assign colours $\mathbf{c}(w_2)\gets 1, \mathbf{c}(w_2,w_3)\gets 2, \mathbf{c}(w_3)\gets 3, \mathbf{c}(w_3,w_4)\gets 4$ and $\mathbf{c}(w_4)\gets 5$. Observe that $w_3$ is a total b-chromatic vertex of colour 3. Next, assign colours $\mathbf{c}(w_1)\gets 4$ and $\mathbf{c}(w_1,w_2)\gets 5$. Note that vertex $w_2$ and edge $(w_2,w_3)$ are total b-chromatic elements of colour 1 and 2, respectively. Next, assign colours $\mathbf{c}(w_4,w_5)\gets 1$ and $\mathbf{c}(w_5)\gets 2$. Note that edge $(w_3,w_4)$ and vertex $w_4$ are total b-chromatic elements of colour 4 and 5, respectively. Lastly, note that every other vertex $w_{i+1}$ and edge $(w_i, w_{i+1})$, for $i=5,\dots,  n-1$, can be coloured with two different colours from the set $\{1,2,3,4,5\}$. Hence, $\varphi_t(P_n)=5$.
\end{proof}

\begin{proof}[Proof of Proposition \ref{proposition:total-b-chromatic-colouring-stars}]
    We will construct a total b-chromatic $m_t(K_{1,n})$-colouring $\mathbf{c}$ of $K_{1,n}$, for $n\geq 1$. Assume that $u$ is the central vertex of $K_{1,n}$ and let $u_1,\dots,u_n$ be the neighbours of $u$. Next, assign colours $\mathbf{c}(u)\gets 1$ and $\mathbf{c}(u,u_i)\gets i+1$, for $i=1,\dots,n$. Note that $u$ and $u_i$ are a total b-chromatic element of colour 1 and $i+1$, for $i=1,\dots,n$, respectively. Lastly, assign colours $\mathbf{c}(u_i)\gets i+2$, for $i=1,\dots,n-1$, and $\mathbf{c}(u_n)\gets 2$ to complete the total colouring. Hence, $\varphi_t(K_{1,n})=m_t(K_{1,n})=n+1$.
\end{proof}

\newpage

\section{Proof omitted in Section \ref{section:np-completeness}} \label{appendix:section:proofs-omitted-np-completeness}

\begin{proof}[Proof of Theorem \ref{theorem:total-b-chromatic-colouring-np-complete}]
    The total degree of each element in $H$ is:
    \begin{itemize}
        \item $d_t(u_j)=8$ for $1\leq j\leq n$,
        \item $d_t(u_j,v)=n+8$ for $1\leq j\leq n$ and $v\in V(H)$,
        \item $d_t(v)=2n+8$,
        \item $d_t(v,v_i)=2n+8$ for $1\leq i\leq 4$,
        \item $d_t(v_i)=2n+8$ for $1\leq i\leq 4$,
        \item $d_t(v_i, u_i^j)=n+5$ for $1\leq i\leq 4$ and $1\leq j\leq n+3$ and
        \item $d_t(u_i^j)=1$ for $1\leq i\leq 4$ and $1\leq j\leq n+3$.
    \end{itemize}
    
    Now, observe that $n+9$ elements of $H$ have total degree at least $n+8$ so $m_t(H)=n+9$. We will prove that $G$ admits a total colouring with 4 colours if and only if $H$ admits a total b-chromatic colouring with $n+9$ colours. 

    Assume first that $G$ admits a total 4-colouring, namely $\mathbf{c}_G$. Next, let $\mathbf{c}_H$ be a total colouring of $H$. Then, make the following assignment of colours to the elements of $H$:
    \begin{itemize}
        \item $\mathbf{c}_H(u_j)=\mathbf{c}_G(u_j)$ for every $1\leq i\leq n$,
        \item $\mathbf{c}_H(u_k,u_{\ell})=\mathbf{c}_G(u_k,u_{\ell})$ for each edge $(u_k,u_{\ell})\in E(G)$ and $1\leq k,\ell, \leq n$,
        \item $\mathbf{c}_H(u_j,v)=j+4$ for $1\leq j\leq n$,
        \item $\mathbf{c}_H(v)=n+5$,
        \item $\mathbf{c}_H(v,v_i)=n+5+i$ for $1\leq i\leq 4$,
        \item $\mathbf{c}_H(v_i)=i$ for $1\leq i\leq 4$,
        \item Let $i\in \{1,2,3,4\}$ be an integer and $a,b,c\in \{1,2,3,4\}\setminus \{i\}$ be three different colours:
        \begin{itemize}
            \item $\mathbf{c}_H(u_i^j)=z_i$ where $z_i$ is an arbitrary element from $\{a,b,c\}$ for $1\leq j\leq n$,
            \item $\mathbf{c}_H(u_i^{n+1})=n+5+a$, $\mathbf{c}_H(u_i^{n+2})=n+5+b$ and $\mathbf{c}_H(u_i^{n+3})=n+5+c$,
            \item $\mathbf{c}_H(v_i,u_i^j)=j+4$ for $1\leq j\leq n$,
            \item $\mathbf{c}_H(v_i,u_i^{n+1})=a$, $\mathbf{c}_H(v_i,u_i^{n+2})=b$, $\mathbf{c}_H(v_i,u_i^{n+3})=c$.
        \end{itemize}

    \end{itemize}

    \begin{figure}[t!]
        \centering
        \resizebox{0.65\linewidth}{!}{
            \begin{tikzpicture}[node distance={15mm}, thick, main/.style = {draw, circle}]
                \node[main, label=$u_1$] at (2,9) (1) {};
                \node[main, label=$u_2$] at (2,7) (2) {};
                \node[main, label=$u_n$] at (2,4) (3) {};
                
                \draw (2,6.5) ellipse (1.5cm and 4cm);    
                \node[text width=0.5cm] at (2,11) {$G$};
                \node [main, label=$v$, label=below:$n+5$] at (6,7 ) (4) {};
                
                \node [main, label=$v_1$, label=below:1] at (13,13) (5) {};
                \node [main, label=$v_2$, label=below:2] at (13,9)   (6) {};
                \node [main, label=$v_3$, label=below:3] at (13,5)   (7) {};
                \node [main, label=$v_4$, label=below:4] at (13,1)   (8) {};
                
                \node [main, label=left:$u_1^1$, label=above:$z_1$]        at (12, 15) (9) {};
                \node [main, label=right:$u_1^n$, label=above:$z_1$]      at (14, 15) (10) {};
                \node [main, label=right:$u_1^{n+1}$, label=above:{\scriptsize $n+7$}] at (15, 14) (11) {};
                \node [main, label=right:$u_1^{n+2}$, label=above:{\scriptsize $n+8$}] at (15, 13) (12) {};
                \node [main, label=right:$u_1^{n+3}$, label=above:{\scriptsize $n+9$}] at (15, 12) (13) {};
                
                \node [main, label=left:$u_2^1$, label=above:$z_2$]        at (12, 11) (14) {};
                \node [main, label=right:$u_2^n$, label=above:$z_2$]        at (14, 11) (15) {};
                \node [main, label=right:$u_2^{n+1}$, label=above:{\scriptsize $n+6$}] at (15, 10) (16) {};
                \node [main, label=right:$u_2^{n+2}$, label=above:{\scriptsize $n+8$}] at (15, 9) (17) {};
                \node [main, label=right:$u_2^{n+3}$, label=above:{\scriptsize $n+9$}] at (15, 8) (18) {};

                \node [main, label=left:$u_3^1$, label=above:$z_3$]        at (12, 7) (19) {};
                \node [main, label=right:$u_3^n$, label=above:$z_3$]        at (14, 7) (20) {};
                \node [main, label=right:$u_3^{n+1}$, label=above:{\scriptsize $n+6$}] at (15, 6) (21) {};
                \node [main, label=right:$u_3^{n+2}$, label=above:{\scriptsize $n+7$}] at (15, 5) (22) {};
                \node [main, label=right:$u_3^{n+3}$, label=above:{\scriptsize $n+9$}] at (15, 4) (23) {};

                \node [main, label=left:$u_4^1$, label=above:$z_4$]        at (12, 3) (24) {};
                \node [main, label=right:$u_4^n$, label=above:$z_4$]        at (14, 3) (25) {};
                \node [main, label=right:$u_4^{n+1}$, label=above:{\scriptsize $n+6$}] at (15, 2) (26) {};
                \node [main, label=right:$u_4^{n+2}$, label=above:{\scriptsize $n+7$}] at (15, 1) (27) {};
                \node [main, label=right:$u_4^{n+3}$, label=above:{\scriptsize $n+8$}] at (15, 0) (28) {};
                
                \node[main, draw=none] at (1,9.5) (30) {};
                \node[main, draw=none] at (1,9) (31) {};
                \node[main, draw=none] at (1,8.5) (32) {};
                
                \node[main, draw=none] at (1,7.5) (33) {};
                \node[main, draw=none] at (1,7) (34) {};
                \node[main, draw=none] at (1,6.5) (35) {};
                
                \node[main, draw=none] at (1,4.5) (36) {};
                \node[main, draw=none] at (1,4) (37) {};
                \node[main, draw=none] at (1,3.5) (38) {};

                \draw (1) -- (30);
                \draw (1) -- (31);
                \draw (1) -- (32);
                
                \draw (2) -- (33);
                \draw (2) -- (34);
                \draw (2) -- (35);
                
                \draw (3) -- (36);
                \draw (3) -- (37);
                \draw (3) -- (38);
                
                \draw (1) to node[sloped, above] {5} (4);
                \draw (2) to node[sloped, above] {6} (4);
                \draw (3) to node[sloped,above] {$n+4$} (4);
                \draw (4) to node[sloped, above] {$n+6$} (5);
                \draw (4) to node[sloped, above] {$n+7$} (6);
                \draw (4) to node[sloped, above] {$n+8$} (7);
                \draw (4) to node[sloped, above] {$n+9$} (8);

                \draw (5) to node[sloped, above, rotate=180] {5} (9);
                \draw (5) to node[sloped, above] {$n+4$} (10);
                \draw (5) to node[sloped, above] {$2$}  (11);
                \draw (5) to node[sloped, above] {$3$}  (12);
                \draw (5) to node[sloped, above] {$4$}  (13);
                
                \draw (6) to node[sloped, above, rotate=180] {5} (14);
                \draw (6) to node[sloped, above] {$n+4$} (15);
                \draw (6) to node[sloped, above] {$1$} (16);
                \draw (6) to node[sloped, above] {$3$} (17);
                \draw (6) to node[sloped, above] {$4$} (18);

                \draw (7) to node[sloped, above, rotate=180] {5} (19);
                \draw (7) to node[sloped, above] {$n+4$} (20);
                \draw (7) to node[sloped, above] {$1$} (21);
                \draw (7) to node[sloped, above] {$2$} (22);
                \draw (7) to node[sloped, above] {$4$} (23);

                \draw (8) to node[sloped, above, rotate=180] {5} (24);
                \draw (8) to node[sloped, above] {$n+4$} (25);
                \draw (8) to node[sloped, above] {$1$} (26);
                \draw (8) to node[sloped, above] {$2$} (27);
                \draw (8) to node[sloped, above] {$3$} (28);
                
                \path (2) -- (3) node[midway, sloped] {$\dots$};
                \path (9) -- (10) node[midway, sloped] {$\dots$};
                \path (14) -- (15) node[midway, sloped] {$\dots$};
                \path (19) -- (20) node[midway, sloped] {$\dots$};
                \path (24) -- (25) node[midway, sloped] {$\dots$};
            \end{tikzpicture}
        }
        \caption{Transformed graph $H$ with a b-chromatic ($n+9$)-colouring.}
        \label{fig:b-totcol-reduction}
    \end{figure}
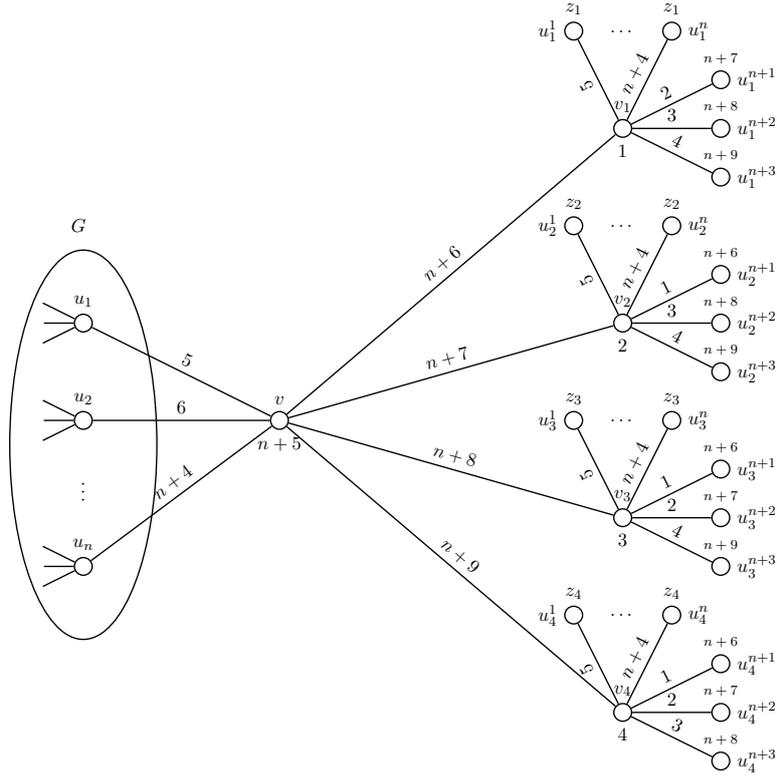
    
    The total colouring $\mathbf{c}_H$ can be observed in Figure \ref{fig:b-totcol-reduction}. It is easy to check that the total colouring is proper. Now, we will prove that $\mathbf{c}_H$ is a total b-chromatic $(n+9)$-colouring. The candidates for becoming total b-chromatic are those elements of $H$ with total degree at least $n+8$: vertices $v_i$ for $1\leq i\leq 4$, edges $(v,v_i)$ for $1\leq i\leq 4$, vertex $v$, and edges $(u_j,v)$ for $1\leq j\leq n$. Vertices $v_i$ for $1\leq i\leq 4$ are picking up colours $\{1,2,3,4\}\setminus \{i\}$ on their neighbours $\{u_i^j\}_{1\leq j\leq n}$, colours $5,\dots, n+4$ on edges $(v_i,u_i^j)$ for $1\leq i\leq 4$ and $1\leq j\leq n$, colour $n+5$ on $v$ and colours $n+5+a, n+5+b$ and $n+5+c$ for $\{a,b,c\}=\{1,2,3,4\}\setminus \{i\}$ on vertices $u_i^{n+1}, u_i^{n+2}$ and $u_i^{n+3}$. Then, vertices $v_i$ are total b-chromatic elements for colour $1\leq i\leq 4$. Now, for every edge $(u_j,v)$ of colour $j+4$, for $1\leq j\leq n$, vertex $u_j$ has colour $\mathbf{c}_H(u_j) \in \{1,2,3,4\}$ and its neighbours $u_j^1,u_j^2,u_j^3$ has colours $\{a,b,c\}=\{1,2,3,4\}\setminus \{\mathbf{c}_H(u_j)\}$ as $G$ admits a total 4-colouring. Moreover, edge $(u_j,v)$ is picking up colours $\{5,\dots,n+4\}\setminus \{j+4\}$ as is sharing vertex $v$ with other edges $(u_{j'},v)$ for $1\leq j'\leq n$ with $j'\neq j$. Furthermore, edges $(u_j,v)$ are picking up colours $n+5$ and $n+5+i$ on vertex $v$ and edges $(v,v_i)$ for $1\leq i\leq 4$. Then edges $(u_j,v)$, for $1\leq j\leq n$, are total b-chromatic elements for colours $5,\dots,n+4$. Next, vertex $v$ of colour $n+5$ is picking up colours $1,\dots,4$ from vertices $v_i$ for $1\leq i\leq 4$, colours $5,\dots, n+4$ from edges $(u_j,v)$ for $1\leq j\leq n$ and colours $n+5+i$ from edges $(v,v_i)$ for $1\leq i\leq 4$ which make vertex $v$ a total b-chromatic element for colour $n+5$. Lastly, edges $(v,v_i)$ for $1\leq i\leq 4$ of colour $n+5+i$ are picking up colours ${a,b,c}=\{1,2,3,4\}\setminus \{i\}$ on edges $(v_i,u_i^{n+1}),(v_i,u_i^{n+2}), (v_i,u_i^{n+3})$ respectively, colours $5,\dots, n+4$ from edges $(v_i,u_i^j)$ for $1\leq j\leq n$, colour $n+5$ from vertex $v$ and colours $n+5+a$, $n+5+b$ and $n+5+c$ for $\{a,b,c\}=\{1,2,3,4\}\setminus \{i\}$ from edges $(v,v_a)$, $(v,v_b)$ and $(v,v_c)$ respectively. Then edges $(v,v_i)$ for $1\leq i\leq 4$ are total b-chromatic elements for colours $n+5+i$.

    Now, let us assume that $H$ admits a total b-chromatic colouring with $n+9$ colours. The only candidates for the total b-chromatic colouring are those elements of $H$ with degree at least $n+8$. These elements are: vertices $v_i$ and edges $(v,v_i)$ for $1\leq i\leq 4$, vertex $v$ and edges $(u_j,v)$ for $1\leq j\leq n$. Now, w.l.o.g. assume that $\mathbf{c}_H(v_i)=i$, $\mathbf{c}_T^{H}(u_j,v)=j+4$, $\mathbf{c}_H(v)=n+5$ and $\mathbf{c}_H(v,v_i)=n+5+i$ for $1\leq i\leq 4$ and for $1\leq j\leq n$. Now, observe that each edge $(u_j,v)$ is picking up colours $j'+4\in \{5,\dots, n+4\}$ for $1\leq j'\leq n$ and $j'\neq j$ on edges $(u_{j'},v)$, colour $n+5$ on $v$ and colour $n+5+i\in \{n+6,\dots, n+9\}$ for $1\leq i\leq 4$ on edges $(v,v_i)$. This implies that $(u_j,v)$ is picking up colours $1,\dots,4$ on vertices $u_j$ and its three incident edges as $G$ is cubic. Moreover, as edge $(u_j,v)$ has total degree $n+8$ it can only pick up colours $1,\dots,4$ on $u_j$ and its neighbours. Furthermore, as $\mathbf{c}_T^{H}$ is a total colouring then $\mathbf{c}_H$ on the induced subgraph $H[\{u_j:1\leq j\leq n\}]$ is a total 4-colouring.
\end{proof}

\newpage

\section{Proofs omitted in Section \ref{section:polynomial-time-algorithm-caterpillars}} \label{appendix:section:proof-section-polynomial-time-algorithms-caterpillars}
\begin{proof}[Proof of Theorem \ref{theorem:total-b-chromatic-colouring-caterpillar-total-m-degree-at-most-5}]
    Assume that the total $m$-degree of $T$ is $m_t(T)=m\leq 5$. We will construct a total b-chromatic colouring $\mathbf{c}$ of $T$ using $m$ colours. Now let $\mathbf{c}:V(T)\cup E(T)\rightarrow [m]$. Next, let $\mathcal{P}$ be the central path of $T$ and assume that every vertex in $\mathcal{P}$ has degree at least 2. Suppose now that there are no vertices of degree at least 3 in $T$. It follows that $T$ is isomorphic to a path $P$, and $\varphi_t(T)=\varphi_t(P)=m_t(P)$ by Proposition \ref{proposition:total-b-chromatic-colouring-paths}. Hence, $T$ must have at least one vertex of degree 3. Furthermore, observe that there can be at most one vertex of degree 4; otherwise, there will be at least 6 elements of total degree at least 5, a contradiction to the fact that $m\leq 5$. Now let $u$ be a vertex of $T$ with $d(u)\in \{3,4\}$. We will split the proof into two cases: 1) $d(u)=3$ and 2) $d(u)=4$.

    \paragraph{Case 1}Assume $u$ is a vertex with $d(u)=3$ and let $u_1,u_2,u_3$ be three neighbours of $u$. Assume w.l.o.g.\ that $u_2$ is a leaf neighbour of $u$. Suppose that no vertex $v\in V(T)$ shares a neighbour with $u$. Then, either $T$ is disconnected, which is a contradiction to our assumption, or $T$ is isomorphic to $K_{1,3}$ and $\varphi_t(K_{1,3})=4$ by Proposition \ref{proposition:total-b-chromatic-colouring-stars}. Hence, there exists a vertex $v\in V(T)$ that share a neighbour with $u$. Assume without loss of generality that $u_1$ is adjacent to $v\neq u$. Now, observe that elements $u,u_1,(u,u_1),(u,u_2),(u,u_3)$ each have total degree at least 4. It follows that $m=5$. Then, assign colours $\mathbf{c}(u)\gets 1$, $\mathbf{c}(u,u_i)\gets i+1$ and $\mathbf{c}(u_i)\gets 5$, for $i=1,2,3$. Observe that elements $u, (u,u_1),(u,u_2), (u,u_3)$ are total b-chromatic elements of colour $1 ,\dots, 4$, respectively. Next, assign  $\mathbf{c}(u_1, v)\gets 3$ and $\mathbf{c}(v)\gets 4$ to make $u_1$ a total b-chromatic element of colour 5.

    Now, we will colour the remaining elements of $T$. Observe that there cannot be a vertex of degree at least 4; otherwise, there will be at least 6 elements of total degree at least 5, which is a contradiction. Similarly, there can be at most two more vertices of degree 3. We will split this case into the subcases: a) there are exactly two vertices of degree 3 in $T$ and b) there are exactly three vertices of degree 3 in $T$. 
    
    \paragraph{Case 1.a}Let $w\neq u$ be the other vertex with $d(w)=3$, and let $w_1,w_2,w_3$ be the three neighbours of $w$ where we assume w.l.o.g.\ that $w_3\not \in P$. We will consider the following three cases: i) $w\in N(u)$, ii) $N(w)\cap N(u)\neq \emptyset$ and $N(w)\cap N(u)=\emptyset$.

    \paragraph{Case 1.a.i}Assume that $w_1=u$. If $w=u_1$ then $w_2=v$. Next, assign $\mathbf{c}(w,w_3)\gets 1$ and $\mathbf{c}(w_3)\gets 2$. Otherwise, $w=u_3$. Assign colours $ \mathbf{c}(w_3)\gets 1, \mathbf{c}(w_2)\gets 2, \mathbf{c}(w,w_2)\gets 1$ and $\mathbf{c}(w,w_3)\gets 2$. 
    
    \paragraph{Case 1.a.ii}Assume w.l.o.g.\ that $N(w)\cap N(u)=\{w_1\}$. If $w=v$ then $w_1=u_1$. Observe that $w_1$ and $(w,w_1)$ are already coloured. Next, assign $\mathbf{c}(w_3)\gets 1, \mathbf{c}(w_2)\gets 2, \mathbf{c}(w,w_2)\gets 1$ and $\mathbf{c}(w,w_3)\gets 2$. Otherwise, $w$ is adjacent to $u_3=w_1$. Then, assign $\mathbf{c}(w)\gets 1, \mathbf{c}(w_2)\gets 2, \mathbf{c}(w_3)\gets 3, \mathbf{c}(w,w_1)\gets 2,\mathbf{c}(w,w_2)\gets 3$ and $\mathbf{c}(w,w_3)\gets 4$.

    \paragraph{Case 1.a.iii}Notice that $w$ may be adjacent to $v$. Next, assign colours $\mathbf{c}(w)\gets 5,\mathbf{c}(w,w_1)\gets 1, \mathbf{c}(w,w_2)\gets 2, \mathbf{c}(w,w_3)\gets 3$. Then, assign colour 4 to the uncoloured neighbours of $w$.

    Now we will colour the uncoloured vertices of $T$, and then colour the uncoloured edges of $T$. Notice that we have coloured every vertex of degree 3 and its incident edges. Therefore, every uncoloured vertex of $x\in V(T)$ has $d(x)\leq 2$. It follows that every pair of adjacent vertices can be assigned different colours. Moreover, every uncoloured edge $e\in E(T)$ is incident to vertices of degree at most 2. Now let $e=(x,y)$ be an uncoloured edge. Since $d(x),d(y)\leq2$, then $e$ is incident to at most 4 different elements. Since we are using 5 colours, it follows that there exists a spare colour for $e$.
    
    \paragraph{Case 1.b}Observe that $\mathcal{P}$ contains only three vertices, each of them of degree 3, otherwise there are at least 6 elements of total degree 5, a contradiction to the fact that $m\leq 5$. Since $d(u)=3$ and $u_2\not \in P$, it follows that $d(u_1)=d(u_3)=3$. Notice that $v\not \in P$. Now let $u_1' \in N(u_1)\setminus \{u,v\}$ and $u_3',u_3''\in N(u_3)\setminus \{u\}$ be the other neighbours of $u_1$ and $u_3$ not in $\mathcal{P}$, respectively. Assign colours $\mathbf{c}(u_1')\gets 3, \mathbf{c}(u_1,u_1')\gets 4$, $\mathbf{c}(u_3')\gets 1,\mathbf{c}(u_3'')\gets 2, \mathbf{c}(u_3,u_3')\gets 2$ and $\mathbf{c}(u_3,u_3'')\gets 1$.

    \paragraph{Case 2}Assume now that $d(u)=4$ and let $u_1,u_2,u_3,u_4$ be the neighbours of $u$. There are five elements in $E(u)\cup \{u\}$ with total degree at least 4. Then, assign colours $\mathbf{c}(u)\gets 1$ and $ \mathbf{c}(u,u_i)\gets i+1$, for $i=1,\dots,4$. Observe that elements in $E(u)\cup \{u\}$ are total b-chromatic elements of colours $1,\dots,5$. Now, it remains to colour the uncoloured elements of $T$. First, assign colours $\mathbf{c}(u_i)\gets (i+2)$, for $i=1,2,3$, and $\mathbf{c}(u_4)\gets 5$. Next, observe that no other vertex $w\in V(T)$ has $d(w)=3$ otherwise there will be at least 6 elements of total degree at least 5, a contradiction to the fact that $m\leq 5$. Therefore, every vertex $w\in V(T)\setminus \{u\}$ has $d(w)\leq 2$. Now, we will colour first the uncoloured vertices of $T$ and then the uncoloured edges. First, since we are using 5 colours, then every pair of adjacent vertices can be assigned different colours. Now, observe that the edges incident to $u$ are already coloured. Now let $e=(x,y)$ be an uncoloured edge. Note that $d(x),d(y) \leq 2$. Then, $e$ is incident to at most 4 different elements. It follows that there exists a spare colour for $e$. Hence, $\varphi_t(T)=5$.
\end{proof}

Before proving Lemma \ref{lemma:total-pivoted-caterpillar-less-than-m} we will introduce two propositions that characterise the structure of total pivoted caterpillars of type 1 and 2.
\begin{proposition}\label{proposition:total-pivoted-tree-type-1-properties}
    Let $T$ be a total pivoted tree of type 1 with $m_t(T)=m\geq 6$. Then, 
    \begin{enumerate}
        \item The set $E(u)\cup \{u,v\}$ contains $m$ total dense elements, and
        \item $d(v)<m-2$.
    \end{enumerate}
\end{proposition}
\begin{proof}
    Let $T$ be a total pivoted tree of type 1 with $m_t(T)=m\geq 6$ and $u$ be a total dense vertex of $T$ with $d(u)=m-2$.
    \begin{enumerate}
         \item Note that $u$ is total dense since $m\geq 6$. Also every edge incident to $u$ has total degree at least $m-1$, so is itself dense.  Thus every one of the $m-1$ elements in $E(u)\cup \{u\}$ is total dense.  Since $T$ contains exactly $m$ total dense elements, and $v$ is total dense, the elements in $E(u)\cup \{u,v\}$ are precisely the total dense elements of $T$.  
         \item Observe that $d(v)<m-2$, for otherwise every edge incident to $v$ is total dense, a contradiction.
    \end{enumerate}
\end{proof}

\begin{proposition}\label{proposition:total-pivoted-tree-type-2-properties}
    Let $T$ be a total pivoted tree of type 2 with $m_t(T)=m\geq 6$. Then,
    \begin{enumerate}
        \item $P_1$ and $P_2$ contains all the total dense elements of $T$ and $m=6$, and
        \item $\ell(P_1)=\ell(P_2)=2$ such that $P_1$ and $P_2$ are total dense paths of type 3,
        \item the total dense vertices and boundary vertices of $P_1$ and $P_2$ has degree two and three, respectively. 
    \end{enumerate}
\end{proposition}
\begin{proof}Let $T$ be a total pivoted tree of type 2 with $m_t(T)=m\geq 6$.
    \begin{enumerate}
        \item If $P_1$ and $P_2$ are the only total dense paths of $T$, then $T$ cannot contain any total dense edge that is not contained in $P_1$ and $P_2$, for otherwise this edge would be part of a third total dense path of length at least 1, a contradiction. Similarly, $T$ cannot contain any total dense vertex that is not contained in $P_1$ or $P_2$; otherwise $v$ would form part of a third total dense path of length at least 0, a contradiction. Hence, $P_1$ and $P_2$ contain all the total dense elements of $T$, and since $P_1$ and $P_2$ each contain three total dense elements, it must be the case that $m=6$.
    
        \item Let $v$ be the vertex that $P_1$ and $Q$ share. Note, that $v$ is a boundary vertex of $P_1$ and $Q$. By definition, $v$ is not total dense. Now, observe that $\ell(P_1)\geq 2$; otherwise, $P_1$ consists of a single edge and cannot have three total dense elements since $v$ is not total dense, a contradiction. Moreover, $\ell(P_1)\leq 2$; otherwise $P_1$ contains at least three total dense edges and at least two total dense vertices, a contradiction to the fact that $m=6$. Hence $\ell(P_1)=2$. Now, assume that $P_1=\langle w_1, w_2,v \rangle$. Since $P_1$ is a total dense path, it follows that $(w_1,w_2),w_2,(w_2,v)$ are total dense. Observe that $v$ cannot be dense because is a boundary vertex of $Q$. Moreover, since $P_1$ contains exactly three total dense elements $w_1$ cannot be dense either. Thus, $P_1$ is a total dense path of type 3. An analogous argument holds for $P_2$. 
    
        \item Since $v$ is adjacent to a vertex in $P_1$ and $P_2$ is not total dense, $d(v)=2$. Furthermore, since $w_2$ is total dense and adjacent to $w_1$ and $v$, then $d(w_2)=3$; otherwise $d(w_2)\geq 4$ and each of its total dense endpoints is total dense, which implies that $m>6$, a contradiction. Lastly, observe that $d(w_1)=2$ since edge $(w_1,w_2)$ is total dense; otherwise $d(w_1)=1$ and edge $(w_1,w_2)$ has total degree $d_t(w_1,w_2)=4$, which is a contradiction. An analogous argument holds for the vertices in $P_2$.
    \end{enumerate}
\end{proof}

\begin{proof}[Proof of Lemma \ref{lemma:total-pivoted-caterpillar-less-than-m}]
    Let $m=m_t(T)$ be the total $m$-degree of $T$ and recall that $T$ has $m$ total dense elements. Assume by a contradiction that $T$ admits a total b-chromatic colouring $\mathbf{c}:V(T)\cup E(T)\rightarrow [m]$. We will consider the two cases under which $T$ is defined.
    
    \paragraph{Case 1}Let $u\in V(T)$ be a total dense vertex with $d(u)=m-2$. Since $m\geq 6$, it follows that $d(u)\geq 4$. It follows that $u$ must be in the central path of $T$. Now let $E(u)=\{(u,u_i):1\leq i\leq m-2\}$ be the set of edges incident to $u$. By Proposition \ref{proposition:total-pivoted-tree-type-1-properties}, it follows that $E(u)\cup \{u,v\}$ contains all the total dense elements of $T$. W.l.o.g.\ assume that $\mathbf{c}(u)=1$, $\mathbf{c}(u,u_i)=i+1$, for $i=1,\dots,m-2$, $\mathbf{c}(v)=m$ and that $u_{m-2}$ is adjacent to $v$. Next, since $(u,u_{m-2})$ is total b-chromatic, it must pick up $m$ colours in its total neighbourhood. Observe that $(u,u_{m-2})$ is picking up colours $1,\dots,m-2$ from elements $(E(u)\setminus \{(u,u_{m-2}\}))\cup \{u\}$. Since $d(u_{m-2})=2$, $(u,u_{m-2})$ must picks up colour $m$ in either $u_{m-2}$ or $(u_{m-2},v)$, which is a contradiction to the proper total colouring. Hence $\varphi_t(T)<m$. An example of a total pivoted caterpillar of type 1 can be observed in Figure \ref{fig:total-pivoted-caterpillar-type-1}.

    \paragraph{Case 2}Let $P_1$ and $P_2$ be two total dense paths of $T$ and let $Q$ be a non-total dense path adjacent to both $P_1$ and $P_2$. Assume that $\langle P_1,Q,P_2\rangle$ is a subpath of the central path of $T$. By definition, $Q$ contains no total dense elements. Now, recall that $P_1$ and $P_2$ are the only total dense paths of $T$. Since both $P_1$ and $P_2$ have exactly three total dense elements each, and $P_1$ and $P_2$ contain all the total dense elements of $T$, it follows that $m=6$. Next, recall that $\ell(Q)\leq 1$. Therefore, we will split the proof into two subcases: a) $\ell(Q)=0$ and b) $\ell(Q)=1$.
    
    \paragraph{Case 2.a}Since $\ell(Q)=0$, $Q$ consists of a single vertex. Let $v$ be the only vertex of $Q$. By Definition \ref{definition:total-pivoted-caterpillar}.2 $v$ is non-total dense. Next, observe that $v$ is a boundary vertex of both $P_1$ and $P_2$. By Proposition \ref{proposition:total-pivoted-tree-type-2-properties}, $\ell(P_1)=\ell(P_2)=2$. Now let $P_1=\langle w_1, w_2,v \rangle$ and $P_2=\langle v, w_3,w_4\rangle$. By Definition \ref{definition:total-dense-path}, elements $(w_1,w_2),w_2,(w_2,v)$ are total dense. By Proposition \ref{proposition:total-pivoted-tree-type-2-properties}, it follows that $w_1$ and $w_4$ cannot be total dense. Similarly, elements $(v,w_3), w_3, (w_3,w_4)$ are the total dense elements of $P_2$. Since $T$ has exactly $m=6$ total dense elements and $\mathbf{c}$ is a total b-chromatic colouring, then elements in $\{(w_1,w_2), w_2, (w_2, v),(v,w_3), w_3, (w_3,w_4)\}$ are the total b-chromatic elements of $T$. An example of a total pivoted caterpillar of type 2 where $\ell(Q)=0$ can be observed in Figure \ref{fig:total-pivoted-caterpillar-type-2-a}.

    Now, assume that $\mathbf{c}(w_1,w_2)=1, \mathbf{c}(w_2)=2, \mathbf{c}(w_2, v)=3, \mathbf{c}(v,w_3)=4, \mathbf{c}(w_3)=5, \mathbf{c}(w_3,w_4)=6$. By Proposition \ref{proposition:total-pivoted-tree-type-2-properties}.3, $d(w_2)=d(w_3)=3$ and $d(w_1)=d(w_4)=d(v)=2$. Next, observe that since $d_t(w_2,v)=5$, $(w_2,v)$ is tight and cannot repeat any colour. It follows that $v$ cannot be assigned colours 1,2 or 3. Similarly, since $d_t(v,w_3)=5$, $(v,w_3)$ is tight and cannot repeat any colour. It follows that $v$ cannot be assigned colours 4,5, or 6. Therefore, $v$ cannot be assigned any colour, which contradicts the fact that $\mathbf{c}$ is a total proper colouring. Hence, $\varphi_t(T)<m$.

    \paragraph{Case 2.b}Since $\ell(Q)=1$, $Q$ consists on a single edge. Let $(u,v)$ be the only edge of $Q$. By Definition \ref{definition:total-pivoted-caterpillar}.2 neither $u$ nor $v$ is total dense. Observe that $u$ is a boundary vertex of $P_1$ and $v$ is a boundary vertex of $P_2$. By Proposition \ref{proposition:total-pivoted-tree-type-2-properties}, $\ell(P_1)=\ell(P_2)=2$. Now let $P_1=\langle w_1, w_2,u \rangle$ and $P_2=\langle v, w_3,w_4\rangle$. By Definition \ref{definition:total-dense-path}, the total dense elements of $P_1$ and $P_2$ are respectively $(w_1,w_2), w_2, (w_2, u)$ and $(v,w_3), w_3, (w_3,w_4)$. Furthermore, by Proposition \ref{proposition:total-pivoted-tree-type-2-properties} $w_1$ and $w_4$ cannot be total dense. Since $T$ has exactly $m=6$ total dense elements and $\mathbf{c}$ is a total b-chromatic colouring, elements in $\{(w_1,w_2), w_2, (w_2, u),(v,w_3), w_3, (w_3,w_4)\}$ are the total b-chromatic elements of $T$. An example of a total pivoted caterpillar of type 2 where $\ell(Q)=1$ can be observed in Figure \ref{fig:total-pivoted-caterpillar-type-2-b}.
    
    Now, assume that $\mathbf{c}(w_1,w_2)=1, \mathbf{c}(w_2)=2, \mathbf{c}(w_2, u)=3, \mathbf{c}(v,w_3)=4, \mathbf{c}(w_3)=5, \mathbf{c}(w_3,w_4)=6$. By Proposition \ref{proposition:total-pivoted-tree-type-2-properties}.3, $d(w_2)=d(w_3)=3$ and $d(w_1)=d(w_4)=d(v)=d(u)=2$. Next, observe that since $d_t(w_2,u)=5$, $(w_2,u)$ is tight and cannot repeat any colour. It follows that $(u,v)$ cannot be assigned colours 1,2 or 3. Similarly, since $d_t(v,w_3)=5$, $(v,w_3)$ is tight and cannot repeat any colour. It follows that $(u,v)$ cannot be assigned colours 4,5, or 6. Therefore, $(u,v)$ cannot be assigned any colour, which is a contradiction. Hence, $\varphi_t(T)<m$.
\end{proof}

\begin{proof}[Proof of Theorem \ref{theorem:total-pivoted-caterpillar-m-1}]
    Let $m=m_t(T)\geq 6$ be the total $m$-degree of $T$. By Lemma \ref{lemma:total-pivoted-caterpillar-less-than-m}, $\varphi_t(T)<m$. We will construct a total b-chromatic $(m-1)$-colouring of $T$. Let $\mathbf{c}:V(T)\cup E(T)\rightarrow [m-1]$. We will split the proof into two cases: 1) $T$ is a total pivoted caterpillar of type 1, and 2) $T$ is a total pivoted caterpillar of type 2.

    \paragraph{Case 1: }Let $u$ be a vertex of $T$ with $d(u)=m-2$ and let $E(u)=\{(u,u_i):1\leq i\leq m-2\}$ be the set of edges incident to $u$. By Definition \ref{definition:total-pivoted-caterpillar}, let $v$ be the total dense vertex of $T$ that shares a neighbour with $u$ and assume w.l.o.g.\ that $N(v)\cap N(u)=\{u_1\}$. Since $T$ is total pivoted, it follows that $d(u_1)=2$. Next, assign colour $\mathbf{c}(u)\gets 1$ and $\mathbf{c}(u,u_i)\gets i+1$, for $i=1,\dots,m-2$. It follows that elements in $E(u)\cup \{u\}$ are total b-chromatic elements of colours $1,\dots,m-1$. Now, it remains to colour the other elements of $T$. First, assign colour $\mathbf{c}(u_i)\gets i+2$, for every $i=1,\dots,m-3$, and $\mathbf{c}(u_{m-2})\gets 2$. Next, assign $\mathbf{c}(u_1,v)\gets 1, \mathbf{c}(v)\gets 2$. By Proposition \ref{proposition:total-pivoted-tree-type-1-properties}, $d(v)<m-1$. Since $(u_1,v)$ is already coloured, it follows that there are at most $m-3$ uncoloured edges in $E(v)$. Then, assign colours $\{3,\dots, m-1\}$ to $(E(v)\setminus \{(u_1,v)\}$ and colour 1 to every vertex in $N(v)\setminus \{u_1\}$. 
    
    Now, since $T$ only have $m$ total dense elements every other vertex on the central path $w\in P\setminus \{u,v\}$ has $d(w)\leq (m-3)/2$. First, since $m\geq 6$, every pair of adjacent vertices can be coloured with two different colours. Now, let $e=(x,y)$ be an uncoloured edge of $T$. Assume first that $e$ is outside the central path. Observe that edge $e$ is adjacent or incident to at most $(m-3)/2-1+1=(m-3)/2$ elements. Assume now that $e$ is in the central path. Observe that $e$ is adjacent or incident to $(m-3)/2-1+(m-3)/2-1+2=m-3$ different elements. It follows that there exists a spare colour for $e$. Hence, $\varphi_t(T)=m-1$.

    \paragraph{Case 2:}Let $P_1$ and $P_2$ be two total dense paths of $T$ and let $Q$ be a non-total dense path adjacent to both $P_1$ and $P_2$. Assume that $\langle P_1,Q,P_2\rangle$ is a subpath of the central path of $T$. Also, $Q$ contains no total dense vertices. Now, recall that $P_1$ and $P_2$ are the only total dense paths of $T$. By Proposition \ref{proposition:total-pivoted-tree-type-2-properties}, $m=6$. Next, recall that $\ell(Q)\leq 1$. Since no element in $Q$ is total dense, we will first make total b-chromatic the total dense elements of $P_1$ and $P_2$, and then colour the elements in $Q$. Lastly, we will colour the remaining elements of $T$. We will split the proof into two subcases: a) $\ell(Q)=0$ and b) $\ell(Q)=1$.
    
    \paragraph{Case 2.a:}By Proposition \ref{proposition:total-pivoted-tree-type-2-properties}, $\ell(P_1)=\ell(P_2)=2$. Now, assume that $P_1=\langle w_1,w_2,v\rangle$, $P_2=\langle v,w_3,w_4\rangle$ and $Q=\langle v\rangle$ where $v$ is a boundary vertex of $P_1$ and $P_2$. By Proposition \ref{proposition:total-pivoted-tree-type-2-properties}.3 $P_1$ and $P_2$ are total dense paths of type 3. It follows that the total dense elements of $P_1$ and $P_2$ are $(w_1,w_2),w_2,(w_2,v)$ and $(v,w_3),w_3, (w_3,w_4)$, respectively. Now, we will construct a total b-chromatic 5-colouring of $T$. First, assign $\mathbf{c}(w_1,w_2)\gets 1, \mathbf{c}(w_2)\gets 2, \mathbf{c}(w_2,v)\gets 3, \mathbf{c}(v,w_3)\gets 4$ and $\mathbf{c}(w_3)\gets 5$. By Proposition \ref{proposition:total-pivoted-tree-type-2-properties}.3, $d(w_2)=d(w_3)=3$ and $d(w_1)=d(w_4)=d(v)=2$. Now let $w_2'\in N(w_2)\setminus \{w_1,v\}$ and $w_3'\in N(w_3)\setminus \{v,w_4\}$ be the neighbours of $w_2$ and $w_3$ not in the central path, respectively. Next, assign colours $\mathbf{c}(w_1)\gets 4, \mathbf{c}(w_2,w_2')\gets 5$ to make $(w_1,w_2), w_2$ and $(w_2,v)$ total b-chromatic elements of colours 1,2 and 3, respectively. Next, assign $\mathbf{c}(w_3,w_4)\gets 1, \mathbf{c}(w_3,w_3')\gets 2$ and $\mathbf{c}(w_4)\gets 3$ to make $(v,w_3)$ and $w_3$ total b-chromatic elements of colour 4 and 5, respectively. Lastly, assign $\mathbf{c}(w_3')\gets 1$ and $\mathbf{c}(v)\gets 1$.

    \paragraph{Case 2.b:}By Proposition \ref{proposition:total-pivoted-tree-type-2-properties}.2, $\ell(P_1)=\ell(P_2)=2$. Now, assume that $P_1=\langle w_1,w_2,u\rangle$, $P_2=\langle v,w_4,w_5\rangle$ and $Q=\langle u, v\rangle$. By Proposition \ref{proposition:total-pivoted-tree-type-2-properties}.2 $P_1$ and $P_2$ are total dense paths of type 3. It follows that the total dense elements of $P_1$ and $P_2$ are $(w_1,w_2),w_2,(w_2,u)$ and $(v,w_3),w_3, (w_3,w_4)$, respectively. Now assign $\mathbf{c}(w_1,w_2)\gets 1, \mathbf{c}(w_2)\gets 2, \mathbf{c}(w_2,u)\gets 3, \mathbf{c}(v,w_3)\gets 4$ and $\mathbf{c}(w_3)\gets 5$. By Proposition \ref{proposition:total-pivoted-tree-type-2-properties}.3, $d(w_2)=d(w_3)=3$ and $d(w_1)=d(w_4)=d(v)=d(u)=2$. Now let $w_2'\in N(w_2)\setminus \{w_1,v\}$ and $w_3'\in N(w_3)\setminus \{u,w_4\}$ be the neighbours of $w_2$ and $w_3$ not in the central path, respectively. Next, assign colours $\mathbf{c}(w_1)\gets 4, \mathbf{c}(w_2,w_2')\gets 5$ and $\mathbf{c}(u)\gets 4$ to make $(w_1,w_2), w_2$ and $(w_2,v)$ total b-chromatic elements of colour 1, 2 and 3, respectively. Next, assign $\mathbf{c}(w_4)\gets 1, \mathbf{c}(w_3,w_4)\gets 2, \mathbf{c}(v)\gets 3$ and $\mathbf{c}(w_3,w_3')\gets 1$ to make $(v,w_3)$ and $w_3$ total b-chromatic elements of colours 4 and 5, respectively. Lastly, assign $\mathbf{c}(w_3')\gets 2$ and $\mathbf{c}(u,v)\gets 1$.
    
    Now it remains to colour the uncoloured elements of $T$. Since $T$ has exactly $m=6$ total dense elements, it follows that every other vertex $w\in V(T)\setminus \{w_2,w_3\}$ has degree $d(w)\leq 2$. Hence, every pair of adjacent vertices can be given different colours. Moreover, it follows that every other uncoloured edge $e$ in $T$ is adjacent to at most two other edges. Together with its two endpoints, $e$ is adjacent or incident to at most 4 elements. It follows that there is a spare colour for $e$. Hence, $\varphi_t(T)=5$.
\end{proof}

The following proposition will be used in the proof of Theorem \ref{theorem:non-total-pivoted-caterpillar-total-element-outside-central-path}.

\begin{proposition}\label{proposition:dense-edge-one-endpoint-dense}
    Let $G$ be a graph with total $m$-degree $m_t(G)=m$ and let $(u,v)$ be an edge of $G$. If $(u,v)$ is a total dense edge then at least one of its endpoints, $u$ or $v$, is total dense.
\end{proposition}
\begin{proof}
    Assume by contradiction that neither $u$ nor $v$ are total dense. This implies that $d(u)< (m-1)/2 $ and $d(v)<(m-1)/2 $ which implies that $d_t(u,v)<m-1$, a contradiction.
\end{proof}

\begin{proof}[Proof of Theorem \ref{theorem:non-total-pivoted-caterpillar-total-element-outside-central-path}]
    Assume that the total $m$-degree of $T$ is $m_t(T)=m\geq 6$. Furthermore, let $\mathcal{P}$ be the central path of $T$ and assume that every vertex in $\mathcal{P}$ has degree at least 2. Suppose first that there are no vertices of degree 3 in $\mathcal{P}$. It follows that $T$ is isomorphic to a path $P$, and $\varphi_t(T)=\varphi_t(P)=m_t(P)$ by Proposition \ref{proposition:total-b-chromatic-colouring-paths}. Therefore, we will assume that there exists at least one vertex of degree 3 in $\mathcal{P}$. Next, since $m\geq 6$, it follows that no vertex outside the central path is total dense.

    Let $e$ be a total dense edge outside the central path $\mathcal{P}$ that is incident to a vertex $u\in \mathcal{P}$. Since $e$ is total dense and its endpoint outside the central path has degree 1, it follows that $d(u)\geq m-2$. Furthermore, $d(u)\leq m-1$ for otherwise, there would be at least $m+1$ elements of total degree at least $m$, a contradiction to the fact that $m_t(T)=m$. Hence, $d(u)\in \{m-1, m-2\}$. Let $E(u)=\{(u,u_i):1\leq i\leq d(u)\}$ be the set of edges incident to $u$. Notice that the set of elements $E(u)\cup \{u\}$ are total dense elements of $T$. We will split the proof into two subcases: 1) $d(u)=m-1$ and 2) $d(u)=m-2$.
    
    \paragraph{Case 1:} Observe that $d_t(u)=2m-2$ and $d_t(u,u_i)=m$, for $i=1,\dots,m-1$. It follows that element $x\in E(u)\cup \{u\}$ has total degree $d_t(x)\geq m$.\ Since $|E(u)\cup \{u\}|=m$, then every vertex $w\in V(T)\setminus \{u\}$ has $d(w)<m/2$; otherwise, there are at least $m+1$ elements with total degree at least $m$, a contradiction to the fact that $m_t(T)=m$. Now, assign $\mathbf{c}(u)\gets 1$ and $\mathbf{c}(u_i)\gets i+1$ for $i=1,\dots,m-1$ to make elements in $E(u)\cup \{u\}$ total b-chromatic elements of colours $1,\dots,m$. 

    Now, we will colour the remaining elements of $T$. We will first colour the vertices and then colour the edges in such a way that the overall total colouring remains proper. First, since we are using $m$ colours, every pair of adjacent vertices can receive two different colours. Now, we will colour the uncoloured edges of $T$. First, assume that $e=(v,w)$ is an edge of $T$ outside the central path. Suppose that $v$ is on the central path. Then $d(v)\leq (m-2)/2$ as observed above. Recall that $d(w)=1$. Hence $e$ is adjacent or incident to at most $(m-2)/2 -1 + 2= m/2$ elements. It follows that there is a spare colour for $e$. Assume now that $e$ is on the central path.  We know that $d(w)\leq (m-2)/2$ and $d(v)\leq (m-2)/2$. It follows that $e$ is adjacent or incident to at most $(m-2)/2 -1 + (m-2)/2 -1  + 2 = m-2$ elements. Hence, there is a spare colour for $e$.
    
    \paragraph{Case 2:} Observe that $d_t(u)= 2m-4$ and $d_t(u,u_i)=m-1$, for $i=1,\dots,m-2$. Furthermore, every vertex $w\in V(T)\setminus \{u\}$ has $d(w)\leq m-2$; otherwise, there are at least $m+1$ elements of total degree at least $m$, a contradiction to the fact that $m_t(T)=m$. Next, observe that $E(u)\cup \{u\}$ contains $m-1$ total dense elements. It follows that there must exist at least one more total dense element in $T$. Now, assign $\mathbf{c}(u)\gets 1$ and $\mathbf{c}(u,u_i)\gets i+1$, for $i=1,\dots,m-2$, so that vertex $u$ pick up colours $\{2,\dots, m-1\}$ and edge $(u,u_i)$ pick up colours $[m-1]\setminus \{i\}$, for every $i=1,\dots,m-2$. Observe that every element in $E(u)\cup \{u\}$ needs to picks up colour $m$ to become total b-chromatic. Furthermore, it remains to find the total b-chromatic element of colour $m$. Next, let $x$ be the $m$-th total dense element of $T$. By Proposition \ref{proposition:dense-edge-one-endpoint-dense}, we can always choose $x$ to be a total dense vertex. Assume that $x=v$ is a total dense vertex of $T$ and assign colour $\mathbf{c}(v)\gets m$.  We will consider three subcases: a) $v\in N(u)$, b) $N(u)\cap N(v)\neq \emptyset$ and c) $N(u)\cap N(v)=\emptyset$.

    \paragraph{Case 2.a:} Observe that $v\in N(u)$. Then $v$ must be on $\mathcal{P}$. Assume w.l.o.g.\ that $v=u_1$. Then, assign colour $m$ to all the vertices in $N(u)\setminus \{v\}$ since $v$ has already colour $m$. It follows that every element in $E(u)\cup \{u\}$ is a total b-chromatic element. Next, assign $\{3,\dots, m-1\}$ to elements in $(N(v)\cup E(v))\setminus \{u,(v,u)\}$ to make $v$ a total b-chromatic vertex of colour $m$.
    
    \paragraph{Case 2.b}Assume w.l.o.g.\ that $u_1\in N(u)\cap N(v)$. Suppose first that $d(u_1)>2$ and let $(u_1,w)$ be an edge incident to $u_1$ where $w\neq v$ and $w\neq u$. Then assign $\mathbf{c}(u_1,w)\gets m$ and $\mathbf{c}(u_i)\gets m$, for $i=2,3,\dots, m-2$. Observe now that elements from $E(u)\cup \{u\}$ are a total b-chromatic elements of colours $1,\dots,m-1$. Now, assign colours $\mathbf{c}(u_1)\gets 3$ and $\mathbf{c}(u_1,v)\gets 1$. Next, assign $\{2,\dots, m-1\}\setminus \{3\}$ to elements in $(N(v)\cup E(v)) \setminus \{u_1,(v,u_1)\}$ to make $v$ a total b-chromatic vertex of colour $m$.

    Suppose now that $d(u_1)=2$. Recall that $\mathbf{c}(u,u_1)=2$. Then, observe that edge $(u,u_1)$ cannot pick up colour $m$ on its neighbourhood since $\mathbf{c}(v)=m$. It follows that $(u,u_1)$ cannot become total b-chromatic. Since $T$ is not total pivoted, it follows that there exists another total dense element $y$. By Proposition \ref{proposition:dense-edge-one-endpoint-dense}, we can always choose $y$ to be a total dense vertex of $T$. Assume that $y=w$ is a total dense vertex. Observe that $w\neq u_1$ since $d(u_1)=2$. Next, we will consider the following three cases: i) $w\in N(u)$ or ii) $w\in N(v)$ or iii) $w\not\in  N(u)\cup N(v)$.

    \paragraph{Case 2.b.i} Since $w\in N(u)$, then $w\in \mathcal{P}$. Assume w.l.o.g.\ that $w=u_{m-2}$. Then assign colours $\mathbf{c}(w)\gets 2, \mathbf{c}(u_1)\gets m-1$ and $\mathbf{c}(u_i)\gets m$, for $i=2,\dots, m-3$. Notice that elements in $(E(u)\setminus \{(u,u_1),(u,w)\})\cup \{u\}$ are total b-chromatic elements of colour $[m-1]\setminus\{2\}$. It remains to make $w$ and $v$ total b-chromatic vertices of colour 2 and $m$. First, observe that $w$ is picking up colours $\mathbf{c}(u)=1$ and $\mathbf{c}(u,w)=m-1$. Next, assign colours $\{3,\dots, m\}\setminus \{m-1\}$ to elements in $(E(w)\cup N(w))\setminus \{(u,w),u\}$ to make $w$ a total b-chromatic element of colour 2. Next, observe that $v$ is picking up colour $m-1$ from $u_1$. Then, assign colours $\mathbf{c}(u_1,v)\gets 1$ and $\{3,\dots, m-2\}$ to elements in $(E(v)\cup N(v))\setminus \{(u_1,v),u_1\}$ to make $v$ a total b-chromatic element of colour $m$.

    \paragraph{Case 2.b.ii} Since $w\in N(v)$, then $w\in \mathcal{P}$. First, assign colours $\mathbf{c}(w)\gets 2, \mathbf{c}(u_1)\gets 3$ and $\mathbf{c}(u_i)\gets m$, for $i=2,\dots, m-2$. Notice that elements in $(E(u)\setminus \{(u,u_1)\})\cup \{u\}$ are total b-chromatic elements for colour $[m-1]\setminus \{2\}$. It remains to make $w$ and $v$ total b-chromatic vertices of colour 2 and $m$. Observe that $v$ and $w$ are picking up colours 2 and $m$ respectively. Next, assign colours $\mathbf{c}(u_1,v)\gets 1, \mathbf{c}(v,w)\gets 4$ and $\{5,\dots, m-1\}$ to elements in $(N(v)\cup E(v)\setminus \{(u_1,v),w, u_1\})$ to make $v$ a total b-chromatic vertex of colour $m$. Lastly, assign colours $[m-1]\setminus \{2,4\}$ to elements in $(E(w)\cup N(w))\setminus \{(v,w),v\}$ to make $w$ a total b-chromatic element of colour 2.

    \paragraph{Case 2.b.iii} Assume now that $w\not \in N(u)\cup N(v)$. Since $w$ is total dense then $w\in \mathcal{P}$. First, assign colours $\mathbf{c}(w)\gets 2,\mathbf{c}(u_1)\gets m-1$ and $\mathbf{c}(u_i)\gets m$, for $i=2,\dots, m-2$. Notice that elements in $(E(u)\setminus \{(u,u_1)\})\cup \{u\}$ are total b-chromatic elements for colour $[m-1]\setminus \{2\}$. It remains to make total b-chromatic vertices $v$ and $w$. First, assign colour $\mathbf{c}(u_1,v)\gets 1$.

    Assume first that $N(w)\cap N(v)\neq \emptyset$ and without loss of generality assume that $w'\in N(w)\cap N(v)$. Next, assign colour $\mathbf{c}(w')\gets 3, \mathbf{c}(w,w')\gets 1$ and $\mathbf{c}(v,w')\gets 2$. Then, assign colours $\{4,\dots, m-2\}$ to elements in $(N(v)\cup E(v))\setminus \{(u_1,v),(v,w'), u_1, w'\}$ to make $v$ a total b-chromatic vertex of colour $m$. Next, assign colours $\{4,\dots m\}$ to elements in $(N(w)\cup E(w))\setminus \setminus \{(w',w),w'\}$ to make $w$ a total b-chromatic vertex of colour 2. Assume now that $N(w)\cap N(u)\neq \emptyset$ and without loss of generality assume that $u_{m-2}\in N(w)\cap N(u)$. Assign colour $\mathbf{c}(w,u_{m-2})\gets 1$ and $\{3,\dots, m\}$ to elements in $(N(w)\cup E(w))\setminus \{(w,u_{m-2}), u_{m-2}\}$ to make $w$ a total b-chromatic vertex of colour 2. Then, assign colour $\{2,\dots, m-1\}$ to elements in $(N(v)\cup E(v))\setminus \{(u_1,v), u_1\}$ to make $v$ a total b-chromatic element of colour $v$. Lastly, assume  that $N(w)\cap N(v)=\emptyset$ and $N(w)\cap N(u)=\emptyset$. Assign colours $[m]\setminus \{2\}$ to elements in $N(w)\cup E(w)$ and colours $[m-1]$ to elements in $N(v)\cap E(v)$ to make $w$ and $v$ total b-chromatic elements of colour 2 and $m$, respectively.
    
    \paragraph{Case 2.c} Observe that $\mbox{dist}(u,v)\geq 3$. Next, assign colour $m$ to all the vertices in $N(u)$ to make elements $E(u)\cup \{u\}$ total b-chromatic elements of colour $1,\dots,m-1$. Now, assign colours $[m-1]$ to elements in $N(v)\cup E(v)$ to make $v$ a total b-chromatic element of colour $m$. 

    Now, we will colour the remaining elements of $T$. Recall that every vertex in $T$ has total degree at most $m-2$. We will first colour edges in the central path $\mathcal{P}$. Observe that every edge $e$ in the central path can be coloured with three different colours from the set $[m]$ since $m\geq 6$. Now let $e$ be an edge outside the central path and assume w.l.o.g.\ that $w\in \mathcal{P}$ is an endpoint of $e$. Assume that $\mathbf{c}(w)=a$ for some $a\in [m]$. Let $E_w=\{(w,w'):w'\not\in P\}$ be the set of edges incident to $w$ outside the central path. Recall that $w$ may be incident to at most two edges in the central path. Assume first that $w$ is only incident to edge $e_1$ of colour $b\in [m]\setminus \{a\}$. Since $w$ is incident one edge in $\mathcal{P}$ then $|E_w|\leq m-3$. Next, assign colour $[m]\setminus \{a,b\}$ to edges in $E_w$ such that each edge get different colours and then assign colour $\mathbf{c}(w')\gets b$, for $w'\in N(w)\setminus \mathcal{P}$. Assume now that $w$ is incident to edges $e_1$ and $e_2$ in $\mathcal{P}$ with colours $\mathbf{c}(e_1)=b$ and $\mathbf{c}(e_2)=c$, where $b,c\in [m]\setminus \{a\}$ and $b\neq c$. Next, since $w$ is incident to two edges in $\mathcal{P}$ then $|E_w|\leq m-4$. Then, assign colours $[m]\setminus \{a,b,c,\}$ to edges in $E(u)$ and colour $\mathbf{c}(w')\gets b$, for every $w'\in N(w)\setminus \mathcal{P}$.
\end{proof}

\begin{proof}[Proof of Lemma \ref{lemma:total-colouring-total-dense-subpath-type-1}]
     Let $m=m_t(T)\geq 6$ be the total $m$-degree of $T$ and let $P=\langle w_1,\dots,w_k\rangle$ be a subpath of the central path of $T$ such that $P$ is a total dense path of type 1. It follows that $w_1$ and $w_k$ are total dense vertices. We will assume that $k$ is odd. The case where $k$ is even follows by an analogous argument. Observe that $P$ has $2k-1$ total dense elements. Furthermore, assume that vertices in $P$ are enumerated from left to right. Thus, vertex $w_{i-1}$ is to the left of $w_i$ and vertex $w_{i+1}$ is to the right of $w_i$ on $P$, for $i=1,\dots,k$. Moreover, let $e_i=(w_i,w_{i+1})$, for $i=1,\dots k-1$, be the total dense edge incident to $w_i$ and $w_{i+1}$. Next, suppose that $w_i$, for some $i\in [k]$, is a total dense vertex with $d(w_i)\leq k-2$. It follows that $d_t(w_i)\leq 2k-4 < q-1\leq m-1$, a contradiction to the fact that $w_i$ is total dense. Hence, if $w_i$ is total dense, then $d(w_i)\geq k-1$.

    Now, let $\mathbf{c}:V(T)\cup E(T)\rightarrow [2k-1]$ be a total colouring given by Algorithm \ref{algorithm:total-colouring-total-dense-path-type-1} executed over $T$ with $P$ as a parameter. We will show that each total dense element $x_i$ in $P$ is picking up colours $[2k-1]\setminus \{\mathbf{c}(x_i)\}$. By Line \ref{alg1:line:1-colouring-vertices} $\mathbf{c}(w_i)=2i-1$, for $i=1,\dots k$. First, we will show that vertex $w_i$, for $i=1,\dots,k$, is picking up colours $[2k-1]\setminus \{2i-1\}$.

    \paragraph{Vertex $w_1$} Since $w_1$ is total dense, it follows $d(w_1)\geq k-1$. Since $w_1$ is only adjacent to vertex $w_2$ in $P$ it follows that $|N(w_1)\setminus \{w_2\}|\geq k-2$. By the assignment made on Lines \ref{alg1:line:1-colouring-vertices}-\ref{alg1:line:2-colouring-edges} it follows that $w_1$ is picking up colours $\mathbf{c}(e_1)=2$ and $\mathbf{c}(w_2)=3$. Now, recall that $k$ is odd. We will partition $N_1$ into $N_1=N_1^{odd}\uplus N_1^{even}$ where
    $$N_1^{odd}=\{w_1^j:j=1,\dots,k-2\} \mbox{ and }N_1^{even}=\{w_1^j:j=2,\dots, k-3\}.$$
    Similarly, we will partition $E_1$ into $E_1=E_1^{odd}\uplus E_1^{even}$ where
    $$E_1^{odd}=\{e_1^j:j=1,\dots,k-2\} \mbox{ and }E_1^{even}=\{e_1^j:j=2,\dots, k-3\}.$$
    Now, observe that Line \ref{alg1:line:5-colouring-w_1^j-e_1^j-odd} assigns colours to $w_1^j\in N_{1}^{odd}$ and $e_1^j\in E_1^{odd}$. Hence $\mathbf{c}(N_1^{odd})=\{4,8,\dots, 2k-2\}$ and $\mathbf{c}(E_1^{odd})=\{5,9,\dots, 2k-1\}$. Similarly, Line \ref{alg1:line:8-colouring-w_1^j-e_1^j-even} assigns colours to $w_1^j\in N_1^{even}$ and $e_1^j\in E_1^{even}$. Observe that $\mathbf{c}(N_1^{even})=\{7,11,\dots,2k-3\}$ and $\mathbf{c}(E_1^{even})=\{6,10,\dots,2k-4\}$. Hence, $\mathbf{c}(N_1\cup E_1)=\{4,\dots, 2k-1\}$. Moreover, observe that $\mathbf{c}(N_1)\cap \mathbf{c}(E_1)=\emptyset$. Therefore, $w_1$ is not repeating a colour on its total neighbourhood.

    \paragraph{Vertex $w_i$ for $i=2,\dots, k-1$} Assume that $i$ is even. The proof for $i$ odd follows by an analogous argument. Since $w_i$ is total dense it follows that $d(w_i)\geq k-1$. Now, observe that $w_i$ is adjacent to $w_{i-1}$ and $w_{i+1}$ and incident to $e_{i-1}$ and $e_i$. By the assignments made in Lines \ref{alg1:line:1-colouring-vertices}-\ref{alg1:line:2-colouring-edges} $w_i$ is picking up colours $2i-3,2i-2,2i,2i+1$. Note that $w_i$ needs to pick up $2k-6$ colours. Now let $N_{\leftarrow i}=\{w_i^j\in N_i:1\leq j\leq i-2\}$ and $N_{i\rightarrow}=\{w_i^j\in N_i:i-1\leq j\leq k-3\}$ be the set of vertices adjacent to $w_i$ and not in $P$ that will be assigned the colours to the ``left'' and ``right'' of $w_i$, respectively. Note that $N_i=N_{\leftarrow i}\uplus N_{i \rightarrow}$ and that $|N_{\leftarrow i}|=i-2$ and $|N_{i\rightarrow}|=k-i-1$. Similarly, let $E_{\leftarrow i}=\{e_i^j\in E_i:1\leq j\leq i-2\}$ and $E_{i\rightarrow}=\{e_i^j\in E_i:i-1\leq j\leq k-3\}$ be the set of edges adjacent to $w_i$ and not in $P$ that will be assigned the colours to the ``left'' and ``right'' of $w_i$, respectively. Note that $E_i=E_{\leftarrow i}\uplus E_{i\rightarrow}$ and that $|E_{\leftarrow i}|=i-2$ and $|E_{i\rightarrow}|=k-i-1$. Observe that $|N_{\leftarrow i}\cup E_{\leftarrow i}|=2i-4$ and $|N_{i\rightarrow }\cup E_{i\rightarrow}|=2k-2i-2$. It follows that $|N_i\cup E_i|=2k-6$.

    Recall that $i$ is even and observe that $i-2$ is even. Next, let $N_{\leftarrow i}=N_{\leftarrow i}^{odd}\uplus N_{\leftarrow i}^{even}$ where 
    $$N_{\leftarrow i}^{odd}=\{w_i^j:j=1,3,\dots,i-3\} \mbox{ and } N_{\leftarrow i}^{even}=\{w_i^j:j=2,4,\dots,i-2\}.$$
    Similarly, let $E_{\leftarrow i}=E_{\leftarrow i}^{odd}\uplus E_{\leftarrow i}^{even}$ where 
    $$E_{\leftarrow i}^{odd}=\{e_i^j:j=1,3,\dots,i-3\} \mbox{ and } E_{\leftarrow i}^{even}=\{e_i^j:j=2,4,\dots,i-2\}.$$
    Observe that Line \ref{alg1:line:15-colouring-w_i^j-e_i^j-odd} assigns colours to $w_i^j\in N_{\leftarrow i}^{odd}$ and $e_i^j\in E_{\leftarrow i}^{odd}$. It follows that $\mathbf{c}(N_{\leftarrow i}^{odd})=\{2i-4,2i-8,\dots, 4\}$ and $\mathbf{c}(E_{\leftarrow i}^{odd})=\{2i-5,2i-9,\dots, 3\}$. Furthermore, $\mathbf{c}(N_{\leftarrow i}^{odd})\cap \mathbf{c}(E_{\leftarrow i}^{odd})=\emptyset$. Next, Line \ref{alg1:line:17-colouring-w_i^j-e_i^j-even} assigns colours to $w_i^j\in N_{\leftarrow i}^{even}$ and $e_i^j\in E_{\leftarrow i}^{even}$. It follows that $\mathbf{c}(N_{\leftarrow i}^{even})=\{2i-7,2i-11,\dots,1\}$ and $\mathbf{c}(E_{\leftarrow i}^{even})=\{2i-6,2i-10,\dots, 2\}$. Furthermore, $\mathbf{c}(N_{\leftarrow i}^{even})\cap \mathbf{c}(E_{\leftarrow i}^{even})=\emptyset$. Hence $\mathbf{c}(N_{\leftarrow i}\cup E_{\leftarrow i})=\{2i-4,\dots, 1\}$. Lastly, observe that $\mathbf{c}(N_{\leftarrow i}^{odd})\cap \mathbf{c}(N_{\leftarrow i}^{even})=\emptyset$ and $\mathbf{c}(E_{\leftarrow i}^{odd})\cap \mathbf{c}(E_{\leftarrow i}^{even})=\emptyset$.
    
    Recall that $k$ is odd and observe that $k-i-1$ is even. Next, let 
    $N_{i\rightarrow}=N_{i\rightarrow}^{odd}\uplus N_{i\rightarrow}^{even}$ where $$N_{i\rightarrow}^{odd}=\{w_i^{i+j-2}:j=1,\dots,k-i-2\} \mbox{ and }N_{i\rightarrow}^{even}=\{w_i^{i+j-2}:j=2,\dots, k-i-1\}.$$
    Similarly, $E_{i\rightarrow}=E_{i\rightarrow}^{odd}\uplus E_{i\rightarrow}^{even}$ where 
    $$E_{i\rightarrow}^{odd}=\{e_i^{i+j-2}:j=1,\dots,k-i-2\} \mbox{ and } E_{i\rightarrow}^{even}=\{e_i^{i+j-2}:j=2,\dots, k-i-1\}.$$
    Observe that Line \ref{alg1:line:22-colouring-w_i^j-e_i^j-odd} assigns colours to $w_i^{i+j-2}\in N_{i\rightarrow}^{odd}$ and $e_i^{i+j-2}\in E_{i\rightarrow}^{odd}$. It follows that $\mathbf{c}( N_{i\rightarrow}^{odd})=\{2i+2,2i+6,\dots, 2k-4\}$ and $\mathbf{c}(E_{i\rightarrow}^{odd})=\{2i+3,2i+7,\dots, 2k-3\}$. Furthermore, $\mathbf{c}(N_{i\rightarrow}^{odd})\cap \mathbf{c}(E_{i\rightarrow}^{odd})=\emptyset$. Then, Line \ref{alg1:line:24-colouring-w_i^j-e_i^j-even} assigns colours to $w_i^{i+j-2}\in  N_{i\rightarrow}^{even}$ and $e_i^{i+j-2}\in E_{i\rightarrow}^{even}$. It follows that $\mathbf{c}(N_{i\rightarrow}^{even})=\{2i+5,2i+9,\dots, 2k-1\}$ and $\mathbf{c}(E_{i\rightarrow}^{even})=\{2i+4,2i+8,\dots, 2k-2\}$. Furthermore, $\mathbf{c}(N_{i\rightarrow}^{even})\cap \mathbf{c}(E_{i\rightarrow}^{even})=\emptyset$. Hence, $\mathbf{c}(N_{i\rightarrow}\cup E_{i\rightarrow})=\{2i+2,\dots, 2k-1\}$. Next, observe that $\mathbf{c}(N_{\rightarrow i}^{odd})\cap \mathbf{c}(N_{\rightarrow i}^{even})=\emptyset$ and $\mathbf{c}(E_{\rightarrow i}^{odd})\cap \mathbf{c}(E_{\rightarrow i}^{even})=\emptyset$. Lastly, observe that $\mathbf{c}(N_{\leftarrow i}\cup E_{\leftarrow i})\cup \mathbf{c}(N_{i\rightarrow}\cup E_{i\rightarrow})=\{1,\dots,2i-4\}\cup \{2i+2,\dots, 2k-1\}$ and $\mathbf{c}(N_{\leftarrow i}\cup E_{\leftarrow i})\cap \mathbf{c}(N_{i\rightarrow}\cup E_{i\rightarrow})=\emptyset$. Therefore, $w_i$ is no repeating colour on its total neighbourhood. 

    \paragraph{Vertex $w_k$} Since $w_k$ is total dense, it follows $d(w_k)\geq k-1$. Since $w_k$ is only adjacent to vertex $w_{k-1}$ in $P$ it follows that $|N(w_k)\setminus \{w_{k-1}\}|\geq k-2$. By the assignment made on Lines \ref{alg1:line:1-colouring-vertices}-\ref{alg1:line:2-colouring-edges} it follows that $w_k$ is picking up colours $\mathbf{c}(e_{k-1})=2k-2$ and $\mathbf{c}(w_{k-2})=2k-3$. We will partition $N_k$ into  $N_k=N_k^{odd}\uplus N_k^{even}$ where
    $$N_k^{odd}=\{w_k^j:j=1,\dots,k-2\} \mbox{ and }N_k^{even}=\{w_k^j:j=2,\dots, k-3\}.$$
    Similarly, we will partition $E_k$ into $E_k=E_k^{odd}\uplus E_k^{even}$ where
    $$E_k^{odd}=\{e_k^j:j=1,\dots,k-2\} \mbox{ and }E_k^{even}=\{e_k^j:j=2,\dots, k-3\}.$$
    Now, observe that Line \ref{alg1:line:6-colouring-w_k^j-e_k^j-odd} assigns colours to $w_k^j\in N_k^{odd}$ and $e_k^j\in E_k^{odd}$. Hence, $\mathbf{c}(N_k^{odd})=\{2,6,\dots, 2k-4\}$ and $\mathbf{c}(E_k^{odd})=\{1,5,\dots, 2k-5\}$. Then, observe that Line \ref{alg1:line:9-colouring-w_k^j-e_k^j-even} assigns colours to $w_i^j\in N_k^{even}$ and $e_i^j\in E_k^{even}$. Observe that $\mathbf{c}(N_k^{even})=\{3,7,\dots, 2k-7\}$ and $\mathbf{c}(E_k^{even})=\{4,8,\dots, 2k-6\}$. Hence $\mathbf{c}(N_k\cup E_k)=\{1,\dots,2k-4\}$. Moreover, observe that $\mathbf{c}(N_k)\cap \mathbf{c}(E_k)=\emptyset$. Therefore, $w_k$ is not repeating a colour on its total neighbourhood.
    
    Next, by Line \ref{alg1:line:2-colouring-edges} $\mathbf{c}(e_i)=2i$, for $i=1,\dots,k-1$. We will now show that edge $e_i$, for $i=1,\dots,k-1$, is picking up colours $[2k-1]\setminus\{2i\}$ on its total neighbourhood. Furthermore, we will show that $e_i$ is not repeating colours.
    
    \paragraph{Edge $e_1$ is picking up $2k-2$ colours} Observe that $e_1=(w_1,w_2)$. From the assignment of colours to the total neighbourhood of $w_1$ it follows that $\mathbf{c}(E_1)=\{5,9,\dots, 2k-1\}\cup \{6,10,\dots, 2k-4\}$. Since $|E_{\leftarrow i}|=i-2$, for $i=2,\dots, k-1$, it follows that $E_{\leftarrow 2}=\emptyset$ and $E_2=E_{2\rightarrow}$. Next, from the assignment of colours to the total neighbourhood of $w_2$ observe that $\mathbf{c}(E_2)=\{7,11,\dots,2k-3\}\cup \{8,12,\dots, 2k-2\}$. Moreover, $\mathbf{c}(E_1)\cap \mathbf{c}(E_2)=\emptyset$. Therefore, $e_1$ is not repeating a colour on its total neighbourhood.

    \paragraph{Edge $e_i$ is picking up $2k-2$ colours, for $i=2,\dots, k-2$} From the assignment of colours on Lines \ref{alg1:line:1-colouring-vertices}-\ref{alg1:line:2-colouring-edges} observe that $e_i$ is picking up colours $2i-2,2i-1,2i+1,2i+2$ from $e_{i-1},w_i,w_{i+1},e_{i+1}$. It remains to prove that $e_i$ is picking up colours $\{1,\dots,2i-3\}\cup \{2i+3,\dots, 2k-1\}$ on elements $E_i\cup E_{i+1}$. Assume w.l.o.g.\ that $i$ is even. It follows that $i+1$ is odd. Furthermore, recall that $E_i=E_{\leftarrow i}\uplus E_{i\rightarrow}$ and, that $E_{\leftarrow i}=E_{\leftarrow i}^{odd}\uplus E_{\leftarrow i}^{even}$ and $E_{i\rightarrow}=E_{i\rightarrow}^{odd}\uplus E_{i\rightarrow}^{even}$. By the assignment of colours made on Lines \ref{alg1:line:15-colouring-w_i^j-e_i^j-odd}, \ref{alg1:line:17-colouring-w_i^j-e_i^j-even}, \ref{alg1:line:22-colouring-w_i^j-e_i^j-odd} and \ref{alg1:line:24-colouring-w_i^j-e_i^j-even}, it follows that
    \begin{align*}
        \mathbf{c}(E_i) &=\mathbf{c}(E_{\leftarrow i}^{odd})\cup \mathbf{c}(E_{\leftarrow i}^{even})\cup \mathbf{c}(E_{i\rightarrow}^{odd})\cup \mathbf{c}(E_{i\rightarrow}^{even})\\
        &=\{2i-5,\dots, 3\}\cup \{2i-6,\dots, 2\}\\
        &\cup \{2i+3,\dots, 2k-3\}\cup \{2i+4,\dots, 2k-2\}
    \end{align*}
    Furthermore, it is easy to check that $\mathbf{c}(E_{\leftarrow i}^{odd})\cap \mathbf{c}(E_{\leftarrow i}^{even})\cap \mathbf{c}(E_{i\rightarrow}^{odd})\cap \mathbf{c}(E_{i\rightarrow}^{even})=\emptyset$. Now let $\ell=i+1$ and observe that $E_{\ell}=E_{\leftarrow \ell}\uplus E_{\ell \rightarrow}$, where $|E_{\leftarrow \ell}|=\ell-2$ and $|E_{\ell\rightarrow}|=k-\ell-1$. Since $\ell$ is odd, it follows that $\ell-2$ and $k-\ell-1$ are both odd. Therefore,
    $$E_{\leftarrow \ell}^{odd}=\{e_{\ell}^j:j=1,\dots,\ell-2\}\mbox{ and }E_{\leftarrow \ell}^{even}=\{e_{\ell}^j:j=2,\dots, \ell-3\}.$$
    
    Line \ref{alg1:line:15-colouring-w_i^j-e_i^j-odd} assigns $\mathbf{c}(e_{\ell}^j)=2\ell-2j-3$ for $e_{\ell}^j\in E_{\leftarrow \ell}^{odd}$. It follows that $\mathbf{c}(E_{\leftarrow \ell}^{odd})=\{2i-3,\dots,1\}$. Similarly, Line \ref{alg1:line:17-colouring-w_i^j-e_i^j-even} assigns $\mathbf{c}(e_{\ell}^j)=2\ell-2j-2$, for $e_{\ell}^j\in E_{\leftarrow \ell}^{even}$. It follows that $\mathbf{c}(E_{\leftarrow \ell}^{even})=\{2i-4,\dots, 4\}$. Furthermore, observe that $\mathbf{c}(E_{\leftarrow \ell}^{odd})\cap \mathbf{c}(E_{\leftarrow \ell}^{even})=\emptyset$. Now, note that 
    $$E_{\ell \rightarrow}^{odd}=\{e_{\ell}^{\ell+j-2}:j=1,\dots,k-\ell-1\}\mbox{ and } E_{\ell\rightarrow}^{even}=\{e_{\ell}^{\ell+j-2}:j=2,\dots, k-\ell-2\}.$$
    Line \ref{alg1:line:22-colouring-w_i^j-e_i^j-odd} assigns $\mathbf{c}(e_{\ell}^{\ell+j-2})=2\ell+2j+1$ for $e_{\ell}^{\ell+k-2}\in E_{\ell \rightarrow}^{odd}$. It follows that $\mathbf{c}(E_{\ell\rightarrow}^{odd})=\{2i+5,\dots 2k-1\}$. Similarly, Line \ref{alg1:line:24-colouring-w_i^j-e_i^j-even} assigns $\mathbf{c}(e_{\ell}^{\ell+j-2})=2\ell+2j$ for $e_{\ell}^{\ell+j-2}\in E_{\ell \rightarrow}^{even}$. It follows that $\mathbf{c}(E_{\ell\rightarrow}^{even})=\{2i+6,\dots, 2k-4\}$. Furthermore, $\mathbf{c}(E_{\ell \rightarrow}^{odd})\cap \mathbf{c}(E_{\ell \rightarrow}^{even})=\emptyset$. By the assignment of colours done before, it follows that 
    \begin{align*}
        \mathbf{c}(E_{\ell})    &=\mathbf{c}(E_{\leftarrow \ell}^{odd})\cup \mathbf{c}(E_{\leftarrow \ell}^{even})\cup  \mathbf{c}(E_{\ell\rightarrow}^{odd})\cup  \mathbf{c}(E_{\ell\rightarrow}^{even})\\
                                &=\{2i-3,\dots,1\}\cup \{2i-4,\dots, 4\}\\
                                &\cup \{2i+5,\dots 2k-1\}\cup \{2i+6,\dots, 2k-4\}.
    \end{align*}
    Hence, $\mathbf{c}(E_i\cup E_{\ell})=\{1,\dots,2i-3\}\cup \{2i+3,\dots, 2k-1\}$. Furthermore, since $\mathbf{c}(E_i)\cap \mathbf{c}(E_{\ell})=\emptyset$ then $e_i$ is not repeating a colour on its total neighbourhood.

    \paragraph{Edge $e_{k-1}$ is picking up $2k-2$ colours} Observe that $e_{k-1}=(w_{k-1},w_k)$. Since $|E_{i\rightarrow}|=k-i-1$, for $i=2,\dots, k-1$, it follows that $E_{k-1\rightarrow}=\emptyset$ and $E_{k-1}=E_{\leftarrow k-1}$. Now, recall that $k$ is odd. It follows that $k-1$ is even. From the assignment of colours made to the total neighbourhood of $w_i$ observe that $\mathbf{c}(E_{k-1})=\{2k-7,2k-11,\dots,3\}\cup \{2k-8,2k-12,\dots, 2\}$. Next, from the assignment of colours made to the total neighbourhood of $w_k$ observe that $\mathbf{c}(E_k)=\{1,5,\dots, 2k-5\}\cup \{4,8,\dots, 2k-6\}$. Moreover $\mathbf{c}(E_{k-1})\cap \mathbf{c}(E_k)=\emptyset$. Therefore, $e_{k-1}$ is not repeating a colour on its total neighbourhood.
\end{proof}

\begin{proof}[Proof of Lemma \ref{lemma:total-colouring-total-dense-subpath-type-2}]
    Assume that $w_k$ is not total dense and recall that $P$ contains $2k-2=q\leq m$ total dense elements. The proof for $w_1$ not being total dense follows an analogous argument. Suppose that $d(w_i)\leq k-2$, for some $i\in [k-1]$. It follows that $d_t(w_i)\leq 2k-4<2k-3=q-1\leq m-1$, a contradiction to the fact that $w_i$ is total dense. Therefore, $d(w_i)\geq k-1$, for $i=1,\dots,k-1$.  By Lines \ref{alg2:line:1-setting-d} and \ref{alg2:line:2-setting-P'}, $P'=\langle w_1,\dots,w_{k'}\rangle$ where $k'=k-1$. By Line \ref{alg2:line:3-getting-colouring-of-P'}, let $\mathbf{c}$ be a total colouring such that elements of $T[P']$ are given colours $1,\dots,2k'-1$. By Lemma \ref{lemma:total-colouring-total-dense-subpath-type-1}, every total dense element in $T[P']$ is picking up $2k'-2=2k-4$ colours. Next, Line \ref{alg2:line:8-assigns-colour-2k-2-to-ek'} assigns colour $\mathbf{c}(e_{k-1})=2k-2$. Observe that vertex $w_{k-1}$ and edge $e_{k-2}$ are picking up colour $2k-2$. Therefore, $w_{k-1}$ and $e_{k-2}$ are picking up $2k-2$ colours. It remains to prove that vertex $w_i$, for $i=1,\dots,k-2$ and edge $e_i$, for $i=1,\dots,k-3$, are picking up colour $2k-2$.

   By the assignment of colours made by Algorithm \ref{algorithm:total-colouring-total-dense-path-type-1}, vertex $w_1$ is picking up $2k-4$ colours from $N_1\cup E_1\cup \{w_2,e_1\}$ where $N_1=\{w_1^j:j=1,\dots,k'-2\}$ and $N_1=\{e_1^j:j=1,\dots,k'-2\}$. Observe that $|N_1\cup \{w_1\}|=k'-1=k-2$ vertices adjacent to $w_1$ are already coloured. Furthermore, $|E_1\cup \{e_1\}|=k'-1=k-2$ edges incident to $w_1$ are already coloured. It follows that there exists at least one uncoloured vertex adjacent to $w_1$ and one uncoloured edge incident to $w_1$, respectively. Let $w_1^{k-1}$ be the uncoloured vertex adjacent to $w_1$ and $e_1^{k-1}$ be the uncoloured edge incident to $w_1$, respectively. If $k$ is even, then Line \ref{alg2:line:21-colour-assignment} assigns colour $\mathbf{c}(e_1^{k-1})=2k-2$. Otherwise, Line \ref{alg2:line:15-colour-assignment} assigns colour $\mathbf{c}(u_1^{k-1})=2k-2$.
    
   By the assignments of colours made by Algorithm \ref{algorithm:total-colouring-total-dense-path-type-1}, vertex $w_i$, for $i=2,\dots, k'-1$, is picking up $2k-4$ colours from $N_i\cup E_i\cup \{w_{i-1},w_{i+1},e_{i-1},e_{i+1}\}$ where $N_i=\{w_i^j:j=1,\dots,k'-3\}$ and $E_i=\{e_i^j:j=1,\dots,k'-3\}$. Observe that $|N_i\cup \{w_{i-1},w_{i+1}|k'-1=k-2$ vertices adjacent to $w_i$ are already coloured. Furthermore, $|E_i\cup \{e_{i-1},e_{i+1}|k'-1=k-2$ edges incident to $w_i$ are already coloured. It follows that there exists at least one uncoloured vertex adjacent to $w_i$ and one uncoloured edge incident to $w_i$, respectively. Let $w_i^{k-1}$ be the uncoloured vertex adjacent to $w_i$ and $e_i^{k-1}$ be the uncoloured edge incident to $w_i$, respectively. If $i$ is even, then Line \ref{alg2:line:21-colour-assignment} assigns colour $\mathbf{c}(e_i^{k-1})=2k-2$. Otherwise, Line \ref{alg2:line:15-colour-assignment} assigns colour $\mathbf{c}(w_i^{k-1})=2k-2$. Hence, $w_i$ is picking up colour $2k-2$ on its total neighbourhood. Lastly, observe that if $i$ is even, then $e_i$ is picking up colour $2k-2$ from $e_i^{k-1}$. Otherwise, $e_i$ is picking up colour $2k-2$ from $e_{i+1}^{k-1}$.
\end{proof}

\begin{proof}[Proof of Lemma \ref{lemma:total-colouring-total-dense-subpath-type-3}]
    First, recall that $P$ contains $2k-3=q\leq m$ total dense elements. Suppose that $d(w_i)\leq k-3$, for some $i\in \{2,\dots, k-1\}$. It follows that $d_t(w_i)\leq 2k-6<2k-4=q-1\leq m-1$, a contradiction to the fact that $w_i$ is total dense. Therefore, $d(w_i)\geq k-2$, for $i=2,\dots, k-1$. By Line \ref{alg3:line:1-setting-P'}, $k'=k-2$ and $P'=\langle u_1,\dots,u_{k'}\rangle$ for $i=1,\dots,k'$. By Line \ref{alg3:line:2-getting-c}-\ref{alg3:line:3-modifying-c}, let $\mathbf{c}$ be a total colouring such that elements of $T[P']$ are given colours $2,\dots,2k-4$. By Lemma \ref{lemma:total-colouring-total-dense-subpath-type-1}, every total dense element in $T[P']$ is picking up $2k'-2=2k-6$ colours. Next, Lines \ref{alg3:line:4-assign-colour-to-e_2^k-2}-\ref{alg3:line:5-assign-colour-to-e_k-1^k-2} assign colours $\mathbf{c}(e_2^{k-2})=2k-3$ and $\mathbf{c}(e_{k-1}^{k-2})=1$. Observe that vertices $w_{k-1}$ and $w_2$ are picking up $2k-4$ colours, respectively. Furthermore, edges $e_{2}$ and $e_{k-2}$ are picking up $2k-4$ colours. We will prove that vertex $w_i$, for $i=3,\dots, k-2$, and edge $e_i$, for $i=3,\dots, k-3$, are picking up colours 1 and $2k-3$.

    By the assignment of colours made by Algorithm \ref{algorithm:total-colouring-total-dense-path-type-3}, vertex $w_i$, for $i=3,\dots, k-3$, is picking up $2k-6$ colours from elements $N_i\cup E_i\cup \{w_{i-1},w_{i+1},e_{i-1},e_{i+1}\}$ where $N_i=\{w_i^j:j=1,\dots,k'-3\}$ and $E_i=\{e_i^j:j=1,\dots,k'-3\}$. Observe that $|N_i\cup \{w_{i-1},w_{i+1}\}|=k'-1=k-3$ vertices adjacent to $w_i$ are coloured. Similarly, $|E_i\cup \{e_{i-1},e_{i+1}\}|=k'-1=k-3$ edges incident to $w_i$ are coloured. It follows that there exists at least one uncoloured vertex adjacent to $w_i$ and one uncoloured edge incident to $w_i$, respectively. Let $w_i^{k-2}$ be the uncoloured vertex adjacent to $w_i$ and $e_i^{k-2}$ be the uncoloured edge incident to $w_i$, respectively. If $i$ is even then Line \ref{alg3:line:12-assign-colour} assign colours $\mathbf{c}(w_i^{k-2})=2k-3$ and $\mathbf{c}(e_{i}^{k-2})=1$. Otherwise, Line \ref{alg3:line:10-assign-colour} assign colours $\mathbf{c}(e_i^{k-2})=2k-3$ and $\mathbf{c}(w_{i}^{k-2})=1$. Observe that $w_i$ is picking up $2k-4$ colours, for $i=3,\dots, k-2$. Next, if $i$ is even, then $e_i$ is picking up colours 1 and $2k-3$ from edges $e_{i}^{k-2}$ and $e_{i+1}^{k-2}$. Otherwise, $e_i$ is picking up colours 1 and $2k-3$ from edges $e_{i+1}^{k-2}$ and $e_{i}^{k-2}$. Hence, $e_i$ is picking up $2k-3$ colours, for $i=3,\dots, k-3$.
\end{proof}

\newpage
\section{Figures omitted in Section \ref{section:polynomial-time-algorithm-caterpillars}} \label{appendix:section:figures}
In this section, we present several figures that were omitted from Section \ref{section:polynomial-time-algorithm-caterpillars}. An example of a total pivoted caterpillar of type 1 can be observed in Figure \ref{fig:total-pivoted-caterpillar-type-1}. Example of total pivoted caterpillar of type 2 where $\ell(Q)=0$ and $\ell(Q)=1$ can be observed in Figure \ref{fig:total-pivoted-caterpillar-type-2-a} and \ref{fig:total-pivoted-caterpillar-type-2-b}, respectively.
\begin{figure}[ht]
    \centering
    \resizebox{.6\linewidth}{!}{
        \begin{tikzpicture}
            \node[draw, circle, fill=black, minimum size=0.5cm, label=above:$u_1$] (a) at (0, 0) {};
            \node[draw, circle, fill=black, minimum size=0.5cm, label=above:$u$, label=below:$1$] (b) at (2.5, 0) {};
            \node[draw, circle, fill=black, minimum size=0.5cm, label=above:$u_7$] (c) at (5, 0) {};
            \node[draw, circle, fill=black, minimum size=0.5cm, label=above:$v$, label=below:$9$] (d) at (7.5, 0) {};
            \node[draw, circle, fill=black, minimum size=0.5cm, label=above:{}] (e) at (10, 0) {};
            \node[draw, circle, fill=black, minimum size=0.5cm, label=below:$u_2$] (f) at (1.5, -2.5) {};
            \node[draw, circle, fill=black, minimum size=0.5cm, label=below:$u_6$] (g) at (3.5, -2.5) {};
            \node[draw, circle, fill=black, minimum size=0.5cm, label=below:{}] (h) at (6.5, -2.5) {};
            \node[draw, circle, fill=black, minimum size=0.5cm, label=below:{}] (i) at (8.5, -2.5) {};
        
            \node at (2.5, -2.5) {\Huge$\cdot\cdot\cdot$};
        
            \node at (-0.9, 0) {\Huge$\cdot\cdot\cdot$};
        
            \node at (10.9, 0) {\Huge$\cdot\cdot\cdot$};
        
            \draw[thick] (a) -- (b) node[midway, above] {2};
            \draw[thick] (b) -- (c) node[midway, above] {8};
            \draw[thick] (c) -- (d);
            \draw[thick] (d) -- (e);
            \draw[thick] (b) -- (f) node[midway, left] {3};
            \draw[thick] (b) -- (g) node[midway, right] {7};
            \draw[thick] (d) -- (h);
            \draw[thick] (d) -- (i);
        \end{tikzpicture}
    }
    \caption{Total pivoted caterpillar of type 1 with $m_t(T)=9$. Note that $d(u)=7$ and $d(v)=4$. Observe that $(u,u_7)$ needs to pick up colour 9 on its total neighbourhood but neither $u_7$ nor $(u_7,v)$ can be assigned colour 9.}
    \label{fig:total-pivoted-caterpillar-type-1}
\end{figure}
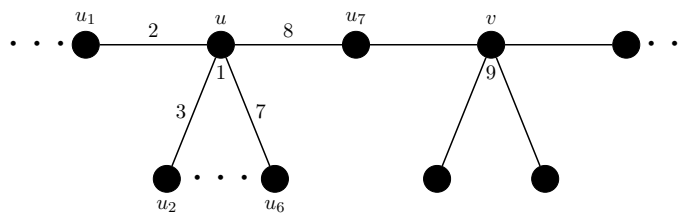

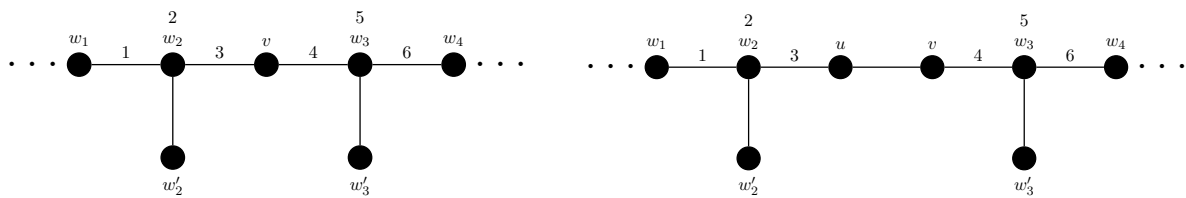
\begin{figure}[ht]
    \centering
    \begin{subfigure}[b]{0.45\textwidth}
        \centering
        \resizebox{1\linewidth}{!}{
            \begin{tikzpicture}[thick]
                \node[draw, circle, fill=black, minimum size=0.5cm, inner sep=0pt, label=above:$w_1$] (w1) at (0, 2) {};
                \node[draw, circle, fill=black, minimum size=0.5cm, inner sep=0pt, label={[yshift=0.5cm]above:$2$}, label=above:$w_2$] (w2) at (2, 2) {};
                \node[draw, circle, fill=black, minimum size=0.5cm, inner sep=0pt, label=above:$v$] (v) at (4, 2) {};
                \node[draw, circle, fill=black, minimum size=0.5cm, inner sep=0pt, label={[yshift=0.5cm]above:$5$}, label=above:$w_3$] (w3) at (6, 2) {};
                \node[draw, circle, fill=black, minimum size=0.5cm, inner sep=0pt, label=above:$w_4$] (w4) at (8, 2) {};
                
                \node[draw, circle, fill=black, minimum size=0.5cm, inner sep=0pt, label=below:$w_{2}'$] (w2a) at (2, 0) {};
                \node[draw, circle, fill=black, minimum size=0.5cm, inner sep=0pt, label=below:$w_{3}'$] (w3a) at (6, 0) {};
                
                \node at (-1, 2) {\Huge$\cdot\cdot\cdot$};
                
                \node at (9, 2) {\Huge$\cdot\cdot\cdot$};
                
                \draw (w1) -- (w2) node[midway, above] {1};
                \draw (w2) -- (v) node[midway, above] {3};
                \draw (v) -- (w3) node[midway, above] {4};
                \draw (w3) -- (w4) node[midway, above] {6};
                
                \draw (w2) -- (w2a);
                \draw (w3) -- (w3a);
                
                \end{tikzpicture}
        }
        \caption{Note that $v$ cannot be assigned a colour in a total 6-colouring of $T$.}
        \label{fig:total-pivoted-caterpillar-type-2-a}
    \end{subfigure}
    \hspace{.2cm}
    \begin{subfigure}[b]{0.45\textwidth}
        \centering
        \resizebox{1.15\linewidth}{!}{
            \begin{tikzpicture}[thick]
                \node[draw, circle, fill=black, minimum size=0.5cm, inner sep=0pt, label=above:$w_1$] (w1) at (0, 2) {};
                \node[draw, circle, fill=black, minimum size=0.5cm, inner sep=0pt, label={[yshift=0.5cm]above:$2$}, label=above:$w_2$] (w2) at (2, 2) {};
                \node[draw, circle, fill=black, minimum size=0.5cm, inner sep=0pt, label=above:$u$] (u) at (4, 2) {};
                \node[draw, circle, fill=black, minimum size=0.5cm, inner sep=0pt, label=above:$v$] (v) at (6, 2) {};
                \node[draw, circle, fill=black, minimum size=0.5cm, inner sep=0pt, label={[yshift=0.5cm]above:$5$}, label=above:$w_3$] (w3) at (8, 2) {};
                \node[draw, circle, fill=black, minimum size=0.5cm, inner sep=0pt, label=above:$w_4$] (w4) at (10, 2) {};
                
                \node[draw, circle, fill=black, minimum size=0.5cm, inner sep=0pt, label=below:$w_{2}'$] (w2a) at (2, 0) {};
                \node[draw, circle, fill=black, minimum size=0.5cm, inner sep=0pt, label=below:$w_{3}'$] (w3a) at (8, 0) {};
                
                \node at (-1, 2) {\Huge$\cdot\cdot\cdot$};
                
                \node at (11, 2) {\Huge$\cdot\cdot\cdot$};
                
                \draw (w1) -- (w2) node[midway, above] {1};
                \draw (w2) -- (u) node[midway, above] {3};
                \draw (u) -- (v);
                \draw (v) -- (w3) node[midway, above] {4};
                \draw (w3) -- (w4) node[midway, above] {6};
                
                \draw (w2) -- (w2a);
                \draw (w3) -- (w3a);
                
                \end{tikzpicture}
        }
        \caption{Note that $(u,v)$ cannot be assigned a colour in a total 6-colouring of $T$.}
        \label{fig:total-pivoted-caterpillar-type-2-b}
    \end{subfigure}
    \caption{Total pivoted caterpillar of type 2.}
    \label{fig:total-pivoted-caterpillar-type-2}
\end{figure}

The colouring produced by Algorithm \ref{algorithm:total-colouring-total-dense-path-type-1} for $k=6$ can be observed in Figure \ref{fig:total-b-chromatic-colouring-running-example-algorithm}.

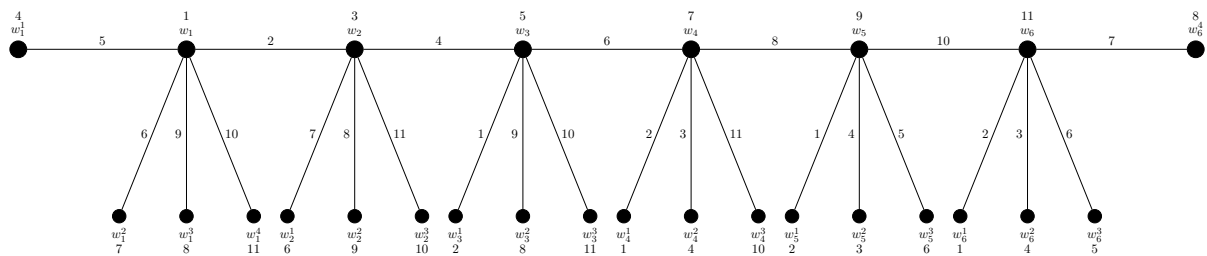
\begin{figure}[ht]
        \centering
        \resizebox{1\linewidth}{!}{
            \begin{tikzpicture}
                \node[draw, circle, fill=black, minimum size=0.5cm, inner sep=0pt, label=above:{$w_1^1$}, label={[yshift=0.5cm]above:{$4$}}] (w1_1) at (-5, 0) {};
                \node[draw, circle, fill=black, minimum size=0.5cm, inner sep=0pt, label=above:{$w_1$}, label={[yshift=0.5cm]above:{$1$}}] (w1) at (0, 0) {};
                \node[draw, circle, fill=black, minimum size=0.5cm, inner sep=0pt, label=above:{$w_2$}, label={[yshift=0.5cm]above:{$3$}}] (w2) at (5, 0) {};
                \node[draw, circle, fill=black, minimum size=0.5cm, inner sep=0pt, label=above:{$w_3$}, label={[yshift=0.5cm]above:{$5$}}] (w3) at (10, 0) {};
                \node[draw, circle, fill=black, minimum size=0.5cm, inner sep=0pt, label=above:{$w_4$}, label={[yshift=0.5cm]above:{$7$}}] (w4) at (15, 0) {};
                \node[draw, circle, fill=black, minimum size=0.5cm, inner sep=0pt, label=above:{$w_5$}, label={[yshift=0.5cm]above:{$9$}}] (w5) at (20, 0) {};
                \node[draw, circle, fill=black, minimum size=0.5cm, inner sep=0pt, label=above:{$w_6$}, label={[yshift=0.5cm]above:{$11$}}] (w6) at (25, 0) {};
                \node[draw, circle, fill=black, minimum size=0.5cm, inner sep=0pt, label=above:{$w_6^4$}, label={[yshift=0.5cm]above:{$8$}}] (w6_4) at (30, 0) {};
            
                \draw[thick] (w1_1) -- (w1) node[midway, above] {$5$};
                \draw[thick] (w1) -- (w2) node[midway, above] {$2$};
                \draw[thick] (w2) -- (w3) node[midway, above] {$4$};
                \draw[thick] (w3) -- (w4) node[midway, above] {$6$};
                \draw[thick] (w4) -- (w5) node[midway, above] {$8$};
                \draw[thick] (w5) -- (w6) node[midway, above] {$10$};
                \draw[thick] (w6) -- (w6_4) node[midway, above] {$7$};
            
                \node[draw, circle, fill=black, minimum size=0.4cm, inner sep=0pt, label=below:{$w_1^2$}, label={[yshift=-0.5cm]below:{$7$}}] (w1_2_below) at (-2, -5) {};
                \node[draw, circle, fill=black, minimum size=0.4cm, inner sep=0pt, label=below:{$w_1^3$}, label={[yshift=-0.5cm]below:{$8$}}] (w1_3_below) at (0, -5) {};
                \node[draw, circle, fill=black, minimum size=0.4cm, inner sep=0pt, label=below:{$w_1^4$}, label={[yshift=-0.5cm]below:{$11$}}] (w1_4_below) at (2, -5) {};
                \draw[thick] (w1) -- (w1_2_below) node[midway, left] {$6$};
                \draw[thick] (w1) -- (w1_3_below) node[midway, left] {$9$};
                \draw[thick] (w1) -- (w1_4_below) node[midway, right] {$10$};
            
                \node[draw, circle, fill=black, minimum size=0.4cm, inner sep=0pt, label=below:{$w_2^1$}, label={[yshift=-0.5cm]below:{$6$}}] (w2_1_below) at (3, -5) {};
                \node[draw, circle, fill=black, minimum size=0.4cm, inner sep=0pt, label=below:{$w_2^2$}, label={[yshift=-0.5cm]below:{$9$}}] (w2_2_below) at (5, -5) {};
                \node[draw, circle, fill=black, minimum size=0.4cm, inner sep=0pt, label=below:{$w_2^3$}, label={[yshift=-0.5cm]below:{$10$}}] (w2_3_below) at (7, -5) {};
                \draw[thick] (w2) -- (w2_1_below) node[midway, left] {$7$};
                \draw[thick] (w2) -- (w2_2_below) node[midway, left] {$8$};
                \draw[thick] (w2) -- (w2_3_below) node[midway, right] {$11$};
            
                \node[draw, circle, fill=black, minimum size=0.4cm, inner sep=0pt, label=below:{$w_3^1$}, label={[yshift=-0.5cm]below:{$2$}}] (w3_1_below) at (8, -5) {};
                \node[draw, circle, fill=black, minimum size=0.4cm, inner sep=0pt, label=below:{$w_3^2$}, label={[yshift=-0.5cm]below:{$8$}}] (w3_2_below) at (10, -5) {};
                \node[draw, circle, fill=black, minimum size=0.4cm, inner sep=0pt, label=below:{$w_3^3$}, label={[yshift=-0.5cm]below:{$11$}}] (w3_3_below) at (12, -5) {};
                \draw[thick] (w3) -- (w3_1_below) node[midway, left] {$1$};
                \draw[thick] (w3) -- (w3_2_below) node[midway, left] {$9$};
                \draw[thick] (w3) -- (w3_3_below) node[midway, right] {$10$};
            
                \node[draw, circle, fill=black, minimum size=0.4cm, inner sep=0pt, label=below:{$w_4^1$}, label={[yshift=-0.5cm]below:{$1$}}] (w4_1_below) at (13, -5) {};
                \node[draw, circle, fill=black, minimum size=0.4cm, inner sep=0pt, label=below:{$w_4^2$}, label={[yshift=-0.5cm]below:{$4$}}] (w4_2_below) at (15, -5) {};
                \node[draw, circle, fill=black, minimum size=0.4cm, inner sep=0pt, label=below:{$w_4^3$}, label={[yshift=-0.5cm]below:{$10$}}] (w4_3_below) at (17, -5) {};
                \draw[thick] (w4) -- (w4_1_below) node[midway, left] {$2$};
                \draw[thick] (w4) -- (w4_2_below) node[midway, left] {$3$};
                \draw[thick] (w4) -- (w4_3_below) node[midway, right] {$11$};
            
                \node[draw, circle, fill=black, minimum size=0.4cm, inner sep=0pt, label=below:{$w_5^1$}, label={[yshift=-0.5cm]below:{$2$}}] (w5_1_below) at (18, -5) {};
                \node[draw, circle, fill=black, minimum size=0.4cm, inner sep=0pt, label=below:{$w_5^2$}, label={[yshift=-0.5cm]below:{$3$}}] (w5_2_below) at (20, -5) {};
                \node[draw, circle, fill=black, minimum size=0.4cm, inner sep=0pt, label=below:{$w_5^3$}, label={[yshift=-0.5cm]below:{$6$}}] (w5_3_below) at (22, -5) {};
                \draw[thick] (w5) -- (w5_1_below) node[midway, left] {$1$};
                \draw[thick] (w5) -- (w5_2_below) node[midway, left] {$4$};
                \draw[thick] (w5) -- (w5_3_below) node[midway, right] {$5$};
            
                \node[draw, circle, fill=black, minimum size=0.4cm, inner sep=0pt, label=below:{$w_6^1$}, label={[yshift=-0.5cm]below:{$1$}}] (w6_1_below) at (23, -5) {};
                \node[draw, circle, fill=black, minimum size=0.4cm, inner sep=0pt, label=below:{$w_6^2$}, label={[yshift=-0.5cm]below:{$4$}}] (w6_2_below) at (25, -5) {};
                \node[draw, circle, fill=black, minimum size=0.4cm, inner sep=0pt, label=below:{$w_6^3$}, label={[yshift=-0.5cm]below:{$5$}}] (w6_3_below) at (27, -5) {};
                \draw[thick] (w6) -- (w6_1_below) node[midway, left] {$2$};
                \draw[thick] (w6) -- (w6_2_below) node[midway, left] {$3$};
                \draw[thick] (w6) -- (w6_3_below) node[midway, right] {$6$};
            \end{tikzpicture}
        }
        \caption{Colouring produced by Algorithm \ref{algorithm:total-colouring-total-dense-path-type-1}. Note that $P'=\langle w_1,\dots,w_6\rangle$ and $m_t(T)=11$.}
        \label{fig:total-b-chromatic-colouring-running-example-algorithm}
\end{figure}  

\newpage
\section{Algorithms omitted in Section \ref{section:polynomial-time-algorithm-caterpillars}}\label{appendix:section:algorithms}
\begin{algorithm}[ht]
    \caption{Algorithm for colouring $T[P]$, where $P=\langle w_1,w_2,\dots, w_k \rangle$ and $k\geq 4$, using $2k-2$ colours}
    \label{algorithm:total-colouring-total-dense-path-type-2}
    \begin{algorithmic}[1]
        \State Let $d=1$ if $w_1$ is not total dense and $d=0$ otherwise \label{alg2:line:1-setting-d}
        \State Let $P'=\langle u_1,\dots,u_{k'}\rangle$ where $u_i=w_{i+d}$ for $i=1,\dots,k'$\Comment{$k'=k-1$}\label{alg2:line:2-setting-P'}
        \State Let $\mathbf{c}$ be a partial total colouring of $T[P']$ produced by Algorithm $\ref{algorithm:total-colouring-total-dense-path-type-1}$\label{alg2:line:3-getting-colouring-of-P'}
        \If{$d=1$}\label{alg2:line:4-if-w_1-is-not-total-dense}
            \State $\mathbf{c}(e_1)\gets 1$\label{alg2:line:5-assigns-colour-1-to-e1}\Comment{assigns colour 1 to the boundary edege $e_1$ of $P$}
            \State $\mathbf{c}(x)\gets 1+\mathbf{c}(x)$ for every $x\in T[P']$ if $\mathbf{c}(x)\neq \emptyset$\Comment{$\mathbf{c}(T[P'])=\{2,\dots, 2k-2$\} if $w_1$ is not total dense}\label{alg2:line:6-changing-colouring-of-P'}
        \Else \label{alg2:line:7}
            \State $\mathbf{c}(e_{k'})\gets 2k-2$\label{alg2:line:8-assigns-colour-2k-2-to-ek'}\Comment{assigns colour $2k-2$ to the boundary edge $e_k$ of $P$}
        \EndIf\label{alg2:line:9}
        \For{$i=1,\dots,k'-1$}\label{alg2:line:10-for-loop-vertex-u_i}
            \If{$i\mod 2=1$}\label{alg2:line:11-if-j-is-odd}
                \If{$d=1$}\Comment{$w_1$ is not total dense}\label{alg2:line:12-if-d-equals-1}
                    \State $\mathbf{c}\left(u_{i+1}^{k'}\right)\gets 1$\label{alg2:line:13-colour-assignment}\Comment{$w_{i+1}$ picks up colour 1}
                \Else \Comment{$w_k$ is not total dense}\label{alg2:line:14-if-d-equals-0}
                    \State $\mathbf{c}\left(u_{k-i-1}^{k'}\right)\gets 2k-2$\label{alg2:line:15-colour-assignment}\Comment{$w_{k-i-1}$ picks up colour $2k-2$}
                \EndIf\label{alg2:line:16}
            \Else\label{alg2:line:17-if-i-is-even}
                \If{$d=1$} \Comment{$w_1$ is not total dense}\label{alg2:line:18-if-d-equals-1}
                    \State $\mathbf{c}\left(e_{i+1}^{k'}\right)\gets 1$\label{alg2:line:19-colour-assignment}\Comment{$e_{i+1}$ picks up colour 1}
                \Else \Comment{$w_k$ is not total dense}\label{alg2:line:20-if-d-equals-0}
                    \State $\mathbf{c}\left(e_{k-i-1}^{k'}\right)\gets 2k-2$\label{alg2:line:21-colour-assignment}\Comment{$e_{k-i-1}$ picks up colour $2k-2$}
                \EndIf\label{alg2:line:22}
            \EndIf\label{alg2:line:23}
        \EndFor\label{alg2:line:24}
    \end{algorithmic}
\end{algorithm}

\begin{algorithm}[ht]
    \caption{Algorithm for colouring $T[P]$, where $P=\langle w_1,w_2,\dots, w_k \rangle$ and $k\geq 5$, using $2k-3$ colours.}
    \label{algorithm:total-colouring-total-dense-path-type-3}
    \begin{algorithmic}[1]
        \State Let $k'=k-2$ and $P'=\langle u_1,\dots,u_{k'}\rangle$ where $u_i=w_{i+1}$ for $i=1,\dots,k'$\Comment{$k'$ is the number of total dense vertices of $P'$}\label{alg3:line:1-setting-P'}
        \State Let $\mathbf{c}$ be a partial total colouring of $T[P']$ produced by Algorithm $\ref{algorithm:total-colouring-total-dense-path-type-1}$\label{alg3:line:2-getting-c}
        \State $\mathbf{c}(x)\gets 1+\mathbf{c}(x)$ for every $x\in T[P']$ if $\mathbf{c}(x)\neq \emptyset$\Comment{$\mathbf{c}(T[P'])=\{2,\dots, 2k-4$\}}\label{alg3:line:3-modifying-c}
        \State $\mathbf{c}\left(e_2^{k-2}\right)\gets 2k-3$\label{alg3:line:4-assign-colour-to-e_2^k-2}\Comment{$e_2$ picks up colour $2k-3$}
        \State $\mathbf{c}\left(e_{k-1}^{k-2}\right)\gets 1$\label{alg3:line:5-assign-colour-to-e_k-1^k-2} \Comment{$e_{k-1}$ picks up colour $1$}
        \For{$i=3,\dots, k-2$}\label{alg3:line:8-for-loop}
            \If{$i\mod 2=1$}\label{alg3:line:9-if-i-is-odd}
                \State $\mathbf{c}\left(w_{i}^{k-2}\right)\gets 1$ and $\mathbf{c}\left(e_{i}^{k-2}\right)\gets 2k-3$\label{alg3:line:10-assign-colour}\Comment{$w_i$ picks up colour 1 and $2k-3$}
            \Else\label{alg3:line:11-if-i-is-even}
                \State $\mathbf{c}\left(w_{i}^{k-2}\right)\gets 2k-3$ and $\mathbf{c}\left(e_{i}^{k-2}\right)\gets 1$\label{alg3:line:12-assign-colour}\Comment{$w_i$ picks up colour $2k-3$ and 1}
            \EndIf\label{alg3:line:13-end-if}
        \EndFor\label{alg3:line:14-end-for}
    \end{algorithmic}
\end{algorithm}

\end{document}